\documentclass[a4paper,UKenglish,cleveref, autoref]{lipics-v2019-arXiv}

\title{Coloring Problems on Bipartite Graphs of Small Diameter} 

\titlerunning{Coloring Problems on Bipartite Graphs of Small Diameter} 

\author{Victor A. Campos}{Departamento de Computa\c c\~ ao, Universidade Federal do Ceará, Fortaleza, Brazil}{campos@lia.ufc.br}{https://orcid.org/0000-0002-2730-4640}{Partially supported by FUNCAP/CNPq/Brazil, Project PRONEM PNE-0112-00061.01.00/16.}

\author{Guilherme C. M. Gomes}{Departamento de Computa\c c\~ ao, Universidade Federal de Minas Gerais, Belo Horizonte, Brazil}{gcm.gomes@dcc.ufmg.br}{https://orcid.org/0000-0002-5164-1460}{Coordenação de Aperfeiçoamento de Pessoal de Nível Superior - Brasil (CAPES) - Finance Code 001.}

\author{Allen Ibiapina}{Departamento de Matemática, Universidade Federal do Ceará, Fortaleza, Brazil}{allen.ibiapina@alu.ufc.br}{[orcid]}{PhD scholarship granted by CAPES.}

\author{Raul Lopes}{Departamento de Computa\c c\~ ao, Universidade Federal do Ceará, Fortaleza, Brazil}{raul.lopes@lia.ufc.br}{https://orcid.org/0000-0002-7487-3475}{PhD scholarship granted by FUNCAP.}

\author{Ignasi Sau}{LIRMM, Université de Montpellier, CNRS, Montpellier, France}{ignasi.sau@lirmm.fr}{https://orcid.org/0000-0002-8981-9287}{CAPES-PRINT Institutional Internationalization Program, process 88887.468331/2019-00, and French projects DEMOGRAPH (ANR-16-CE40-0028), ESIGMA (ANR-17-CE23-0010), and ELIT (ANR-20-CE48-0008-01).}

\author{Ana Silva}{Departamento de Matemática, Universidade Federal do Ceará, Fortaleza, Brazil}{anasilva@mat.ufc.br}{https://orcid.org/0000-0001-8917-0564}{Partially supported by FUNCAP/CNPq/Brazil, Project PRONEM PNE-0112-00061.01.00/16, CNPq Universal 401519/2016-3 and 437841/2018-9, and CNPq Produtividade 304576/2017-4. }

\authorrunning{V. A. Campos, G. C. M. Gomes, A. Ibiapina, R. Lopes, I. Sau, and A. Silva}

\Copyright{V. A. Campos, G. C. M. Gomes, A. Ibiapina, R. Lopes, I. Sau, and A. Silva} 



\category{} 

\relatedversion{} 

\supplement{}

\nolinenumbers 

\hideLIPIcs  

\EventEditors{}
\EventNoEds{2}
\EventLongTitle{}
\EventShortTitle{MFCS 2020}
\EventAcronym{MFCS}
\EventYear{2016}
\EventDate{December 24--27, 2016}
\EventLocation{Little Whinging, United Kingdom}
\EventLogo{}
\SeriesVolume{}
\ArticleNo{1}

\usepackage{url}

\usepackage{tikz,cite}
\usetikzlibrary{positioning,fit,calc,shapes,backgrounds}
\usetikzlibrary{matrix}
\usepackage{tkz-berge}

\usepackage{complexity}
\usepackage{environ,multirow}
\usepackage{setspace}
\usepackage[normalem]{ulem}

\usepackage{hyperref,xspace}
\newcommand{\NPc}{{\sf NP\text{-}c}\xspace}
\renewenvironment{proof}[1][]{\par \noindent {\bf Proof:#1}\ }{\hfill$\Box$\medskip}

\makeatletter
\newcommand{\problemtitle}[1]{\gdef\@problemtitle{#1}}
\newcommand{\probleminput}[1]{\gdef\@probleminput{#1}}
\newcommand{\problemquestion}[1]{\gdef\@problemquestion{#1}}
\NewEnviron{myproblem}{
  \problemtitle{}\probleminput{}\problemquestion{}
  \BODY
  \par\addvspace{.5\baselineskip}
  \noindent
  \begin{tabularx}{\textwidth}{@{\hspace{\parindent}} l X c}
    \multicolumn{2}{@{\hspace{\parindent}}l}{\@problemtitle} \\
    \textbf{Input:} & \@probleminput \\
    \textbf{Question:} & \@problemquestion
  \end{tabularx}
  \par\addvspace{.5\baselineskip}
}

\newcommand{\kLCol}[1]{\textsc{$#1$-List Coloring}}
\newcommand{\LkCol}[1]{\textsc{List $#1$-Coloring}}
\newcommand{\kPre}[1]{\textsc{$#1$-PreExt}}
\newcommand{\kFall}[1]{\textsc{$#1$-Fall-Coloring}}

\newcommand{\RetF}[1]{\textsc{Retract to $#1$}}
\newcommand{\Comp}{\textsc{Edge-Surjective $C_6$-Homomorphism}}
\newcommand{\HComp}{\textsc{Edge-Surjective $H$-Homomorphism}}
\newcommand{\Hom}{\textsc{Surjective $C_6$-Homomorphism}}
\newcommand{\HHom}{\textsc{Surjective $H$-Homomorphism}}
\newcommand{\yes}{{\sc yes}}
\newcommand{\no}{{\sc no}}

\colorlet{myblue}{blue!80}

\hypersetup{
colorlinks = true,
urlcolor={myblue}, 
linkcolor = black!30!blue,
citecolor = black!30!green
}

\EventNoEds{2}
\EventLongTitle{}
\EventShortTitle{arXiv preprint}
\EventAcronym{}
\EventYear{}
\EventDate{}
\EventLocation{}
\EventLogo{}
\SeriesVolume{}
\ArticleNo{}

\begin{document}

\maketitle
\begin{abstract}We investigate a number of coloring problems restricted to bipartite graphs with bounded diameter.
First, we investigate the $k$-\textsc{List Coloring}, \textsc{List $k$-Coloring}, and $k$-\textsc{Precoloring Extension} problems on bipartite graphs with diameter at most $d$, proving $\NP$-completeness in most cases, and leaving open only the  \textsc{List $3$-Coloring} and $3$-\textsc{Precoloring Extension} problems when $d=3$.

Some of these results are obtained through a proof that the \textsc{Surjective $C_6$-Homomorphism} problem is $\NP$-complete on bipartite graphs with diameter at most~four.
Although the latter result has been already proved~[Vikas, 2017], we present ours as an alternative simpler one.
As a byproduct, we also get that \textsc{$3$-Biclique Partition} is $\NP$-complete. An attempt to  prove this result was presented in~[Fleischner, Mujuni, Paulusma, and Szeider, 2009], but there was a flaw in their proof, which we identify and discuss here.

Finally, we prove that the $3$-\textsc{Fall Coloring} problem is $\NP$-complete on bipartite graphs with diameter at most~four, and prove that $\NP$-completeness for diameter three would also imply $\NP$-completeness of $3$-\textsc{Precoloring Extension} on diameter three, thus closing the previously mentioned open cases. This would also answer a question posed in~[Kratochvíl, Tuza, and Voigt, 2002].
\end{abstract}

\section{Introduction}
\label{sec:intro}

Graph coloring problems are among the most fundamental and studied problems in graph theory, due to their practical and theoretical importance.
A \emph{proper coloring} of a graph $G$ is a function $f : V(G) \to \mathbb{N}$ such that $f(u) \neq f(v)$ for every edge $uv \in E(G)$, and the $k$-\textsc{Coloring} problem asks whether a given graph $G$ admits a proper coloring using at most $k$ colors.
A well-known result by Karp~\cite{Karp72} shows that the $k$-\textsc{Coloring} problem is $\NP$-complete for every fixed $k \geq 3$.
In this paper, we study two of the most general coloring problems: \emph{list coloring} and \emph{graph homomorphism}.

In the $k$-\textsc{List Coloring} problem, we are given a graph $G$ together with a function $L$ which assigns to each $u\in V(G)$ a subset of allowed colors with $|L(u)| \leq k$ for every $u \in V(G)$.
This is called a \emph{list assignment of $G$}.
The question is whether $G$ admits a proper coloring $f$ such that $f(u)\in L(u)$ for every $u\in V(G)$; if the answer is \yes, we say that $G$ is \emph{$L$-colorable}.
If $G$ is $L$-colorable for every list assignment $L$ satisfying $|L(v)| \ge k$  for every $v \in V(G)$, we say that $G$ is \emph{$k$-choosable}.

Observe that this generalizes the $k$-\textsc{Coloring} problem: it suffices to consider $L(u) =\{1,\ldots,k\}$ for every $u \in V(G)$.
Thus $k$-\textsc{List Coloring} is $\NP$-complete for every fixed $k\ge 3$; this was shown by Holyer~\cite{H.81} and Karp~\cite{Karp72}.
Another natural coloring problem that can be modeled as a list coloring problem is the \textsc{$k$-Precoloring Extension} problem, where some vertices have fixed colors and the goal is to extend this precoloring to a proper $k$-coloring of $G$.
Thus we can model this problem as a list coloring problem by assigning a list of size $1$ to each of the precolored vertices, and a list equal to $\{1, \ldots, k\}$ to the remaining ones. From now on, we denote this problem by \kPre{k}.

A long-standing and still open question about coloring problems is whether one can decide in polynomial time if a graph with diameter two can be properly colored using at most three colors (see e.g. \cite{K.11,P.16}), and only recently the answer for the $3$-\textsc{Coloring} problem on graphs with diameter at most three has been settled negatively by Mertzios and Spirakis~\cite{MS.16}. Here we propose the investigation of analogous questions concerning list colorings of bipartite graphs.

The famous Hajós Theorem~\cite{H.61} states that every non-$k$-colorable graph can be constructed from the complete graph $K_{k+1}$ by iteratively applying one of three defined operations. Gravier~\cite{G.96} has proved an analogous theorem, showing that every non-$k$-choosable graph can be constructed from complete bipartite graphs by iteratively applying one of three operations (the same as Hajós' operations, with the exception of one, which is an adaptation  of the corresponding Hajós operation  to bipartite graphs).
This is one of the reasons why it can be of interest to investigate list colorings of bipartite graphs (see, for example, \cite{AT.92,K.93,HJP.15,HT.96}); in particular, characterizations of non-$k$-choosable complete bipartite graphs for certain values of $k$ have also been given  (see e.g.~\cite{MRS.91,Donnel,ST.95}).

For graphs $G$ and $H$, we say that $G$ is \emph{$H$-free} if $G$ has no copy of $H$ as an induced subgraph.
Interestingly enough, even though there are some results concerning list coloring problems on $P_k$-free bipartite graphs (see, for example,~\cite{HJP.15,K.93}), 
to the best of our knowledge, no work deals with bipartite graphs of small diameter directly, except for a result by Jansen and Scheffler~\cite{JS.97} that proves that $3$-\textsc{List Coloring} is $\NP$-complete on complete bipartite graphs.
Concerning $P_k$-free graphs, Kratochvíl~\cite{K.93} proved that \kPre{5} is $\NP$-complete on $P_{13}$-free bipartite graphs, while Huang et al.~\cite{HJP.15} showed that \kPre{4} is $\NP$-complete on $P_{10}$-free chordal bipartite graphs and that \kLCol{4} is $\NP$-complete on $P_8$-free chordal bipartite graphs.
Note that if $G$ has no induced $P_k$, then $G$ has diameter at most $k-2$.
Therefore, the aforementioned results of Huang et al.~\cite{HJP.15} give us that \kPre{4} is $\NP$-complete on chordal bipartite graphs with diameter at most~$8$, and that \kLCol{4} is $\NP$-complete on chordal bipartite graphs with diameter at most~$6$.
Here, we improve the first result of Huang et al.~\cite{HJP.15} with respect to the diameter by showing that \kPre{3} is $\NP$-complete on bipartite graphs with diameter~four. Also, we give a linear algorithm that solves \kPre{k}, for each fixed $k$, on complete bipartite graphs (diameter two). Our algorithm is an improvement for complete bipartite over previous more general algorithms that work on $P_5$-free~graphs~(Ho\`ang et al.~\cite{HKLSS.10}), and on $(rP_1 + P_5)$-free graphs~(Couturier et al.~\cite{CGKP.15}); in particular, ours is an $\FPT$ algorithm parameterized by $k$, while the latter algorithms are \XP.
As we will see, this leaves as the only open cases the complexity of \kPre{3} and related problems, when restricted to bipartite graphs with diameter three.

It is well-known that a $k$-coloring can also be seen as a homomorphism to $K_k$ (the complete graph on $k$ vertices).
Given graphs $G$ and $H$, a \emph{homomorphism from $G$ to $H$} is a function $f:V(G)\rightarrow V(H)$ that respects edges, i.e., such that $f(u)f(v)\in E(H)$ whenever $uv\in E(G)$.
When $H$ is fixed, the $H$-\textsc{Homomorphism} problem consists in deciding whether $G$ has a homomorphism to $H$, while the \textsc{List $H$-Homomorphism} is defined in a similar way as the $k$-\textsc{List Coloring} problem, i.e. each vertex $u$ of $G$ can only be mapped to $x\in V(H)$ if $x$ belongs to the list assigned to $u$.
Hell and Ne\v{s}et\v{r}il~\cite{HN.90} proved that \textsc{$H$-Homomorphism} is polynomial if $H$ is bipartite, and $\NP$-complete otherwise.
A dichotomy is also known for the \textsc{List $H$-Homomorphism} problem: Feder and Hell~\cite{FH.98} proved that if $H$ is a \emph{reflexive} graph (a graph is reflexive if every vertex of $H$ has a loop), then the problem is solvable in polynomial time if $H$ is a bipartite interval graph and $\NP$-complete otherwise.
Furthermore, for loopless graphs, Feder et al.~\cite{FHH.99} also proved that if $H$ is bipartite and $\overline{H}$, the \emph{complement} graph of $H$, is a circular-arc graph, then \textsc{List $H$-Homomorphism} can be solved in polynomial time; otherwise, the problem is $\NP$-complete.

A homomorphism $f$ from a graph $G$ to a graph $H$ is \emph{surjective} if every vertex $x \in V(H)$ is the image of some vertex $u \in V(G)$, and it is \emph{edge-surjective} if every edge $xy \in E(H)$ is the image of some edge $uv \in V(G)$.
Observe that if $f$ is edge-surjective and $H$ has no isolated vertices, then $f$ is also surjective.
 If $H$ is a subgraph of $G$ and $f$ is an homomorphism from $G$ to $H$ such that $f(v) = v$ for every $v \in V(H)$, we say that $f$ is a \emph{retraction of $G$ to $H$}.
We denote the related problems by \HHom, \HComp, and \textsc{Retract}, respectively.
Some authors use the term \emph{(edge-)compaction} to refer to (edge-)surjective homomorphism, but we prefer to use a more descriptive term.
Also, if \textsc{Retract} is restricted to instances $(G,H)$ such that $H$ is isomorphic to a fixed graph $F$, then we write \textsc{Retract to $F$} to denote the related problem.

Vikas~\cite{V.04} showed that problem {\HHom} reduces to {\HComp} and to \textsc{Retract to $H$}.
Therefore, it follows from the result by Hell and Nešetřil~\cite{HN.90} that all three problems are $\NP$-complete when $H$ is not bipartite, as it suffices to add a dummy copy of $H$ to $G$ so as to force the obtained graph $G'$ to have a surjective homomorphism to $H$ if and only if $G$ has a homomorphism to $H$. As for the remaining cases, since the late 1980's there has been interest in the complexity of these problems when $H$ is a bipartite graph, with related questions being posed by many authors (see e.g.~\cite{MP.15,V.17}).
Let $C_4^\odot$ denote the reflexive cycle on~four vertices.
Vikas~\cite{V.99} proved that \textsc{Edge-Surjective $C_4^\odot$-Homomorphism} is $\NP$-complete, and only recently \textsc{Surjective $C_4^\odot$-Homomorphism} has been also proved  to be $\NP$-complete by Martin and Paulusma~\cite{MP.15}. Golovach et al.~\cite{GMP.12} proved that {\HHom} is polynomial when $H$ is a path, and $\NP$-complete for many other cases (e.g. linear forests and trees of pathwidth at most~2).
Also, Golovach et al.~\cite{GJMPS.19} recently proved that {\HHom} is $\NP$-complete when $H$ has exactly $2$ vertices $u$ and $v$ such that $\{uu,vv\}\subseteq E(H)$, provided $uv\notin E(H)$.
We refer the reader to the nice survey by Bodirsky et al.~\cite{BKM.12} on surjective homomorphisms and related problems.
Martin and Paulusma~\cite{MP.15} also relate {\HComp} to many other problems (e.g. vertex cut-sets, $H$-partitions, and biclique cover), showing that various open problems and \textsc{Surjective $C_4^\odot$-Homomorphism}, which they prove to be \NP-complete, are equivalent.

In particular,  the $\NP$-completeness of \textsc{Retract to $C_6$} (Feder et al.~\cite{FHH.99}), and of {\Comp} (Vikas~\cite{V.99}) are known since 1999, but only recently the $\NP$-completeness of {\Hom} has been settled (Vikas~\cite{V.17}), even though this had been asked by Hell and Ne\v{s}et\v{r}il back in 1988~\cite{V.17}.
Here we give a stronger  $\NP$-completeness result for \textsc{Retract to $C_6$} that we then use in some of our proofs.
Namely, letting $H \cong C_6$ and $G=(X\cup Y,E)$ be bipartite, we show that  \textsc{Retract to $C_6$} is $\NP$-complete even if $V(H)\cap Y$ dominates $X$, and each $y \in V(H) \cap Y$ is at distance at most~2 from $y'$, for every $y'\in Y$.
We then use this result to show that {\kPre{3}}, {\Comp}, and {\Hom} are $\NP$-complete even when restricted to bipartite graphs with diameter four.
Our $\NP$-completeness proof for {\Hom} produces a smaller graph than the one presented by Vikas~\cite{V.17}, needing only a linear number of new vertices in the construction against a quadratic number used by Vikas~\cite{V.17}, which allows us to present a simpler proof. In addition, we make a small observation that, by letting $M_k$ denote the graph obtained from the complete bipartite graph $K_{k,k}$ by removing a perfect matching, we get that \textsc{Surjective $M_k$-Homomorphism}, \textsc{Edge-Surjective $M_k$-Homomorphism}, and \textsc{Retract to $M_k$} are all $\NP$-complete on graphs of diameter three.

Vikas~\cite{V.17} also points out, without explicitly proving it, that his reduction can be generalized to a proof that \textsc{Surjective $C_{2k}$-Homomorphism} is $\NP$-complete for every $k\ge 3$.
Unfortunately, we have not managed to directly adapt our reduction to this case.
Finally, we mention that there is still interest in the $C_k$-homomorphism problem for odd $k$ when restricted to $P_\ell$-free graphs because of its ties with the long-standing open problem of 3-coloring $P_\ell$-free graphs (see e.g. a recent paper by Chudnovsky et al.~\cite{CHOSZ.20}).

Given a graph $G$, a \emph{biclique of $G$} is a pair of vertex subsets $(A,B)$ forming a complete bipartite subgraph of $G$.
The \emph{bipartite complement} of a bipartite graph $G=(A\cup B, E)$ is the bipartite graph $\overline{G}_{\cal B}=(A\cup B,E')$ containing the non-edges between $A$ and $B$.
In the $k$-\textsc{Biclique} problem, the task is to decide whether $V(G)$ can be partitioned into $k$ bicliques. 
Mohammad et al.~\cite{HMSS.07} showed that, letting $k$ be part of the input, $k$-\textsc{Biclique} is $\NP$-complete even when restricted to bipartite graphs.
An attempt to show that $3$-\textsc{Biclique} is $\NP$-complete on bipartite graphs was presented by Fleischner et al.~\cite{FMPS.09}.
Unfortunately, as we show in Section~\ref{sec:3biclique}, there is a mistake in their proof.
Nevertheless, applying our result for \textsc{Surjective $C_6$-Homomorphism}, we get that the $3$-\textsc{Biclique} problem is indeed $\NP$-complete on bipartite graphs, even when the bipartite complement of $G$ has diameter~four.
It is worth mentioning that Martin and Paulusma~\cite{MP.15} showed that $2$-\textsc{Biclique} is $\NP$-complete on general graphs (with a very technical reduction that uses an abstract algebraic meta-theorem), and that Fleischner et al.~\cite{FMPS.09} showed that it is polynomial on bipartite graphs.
Hence, the \NP-completeness of $3$-\textsc{Biclique} on bipartite graphs is best possible in terms of $k$.

We need two more definitions before presenting our last results.
Given a proper $k$-coloring $f$ of a graph $G$, a vertex $v$ is called a \emph{b-vertex} if the neighborhood of $v$ contains one vertex of each color (distinct from that of $v)$, and $f$ is a \emph{$k$-fall-coloring of $G$} if every vertex of $G$ is a b-vertex.
In the $k$-\textsc{Fall Coloring} problem, we ask whether an input graph $G$ admits a $k$-fall-coloring.
We show that if $3$-\textsc{Fall Coloring} were $\NP$-complete on bipartite graphs with diameter three, then we would get a complete dichotomy for the list coloring problems on bipartite graphs with diameter constraints. Also, this would answer a question posed by Kratochvíl et al.~\cite{KTV.02}.
Although we do not know whether $3$-\textsc{Fall Coloring} is $\NP$-complete on bipartite graphs with diameter three, we show that it is $\NP$-complete on bipartite graphs with diameter four, strengthening a result of Laskar and Lyle~\cite{LL.09}.

\smallskip

\noindent\textbf{Organization}.  In~Section~\ref{sec:preliminaries} we present the formal notation and definitions used in this article.
In Section~\ref{sec:kvsdiam} we give an almost complete classification for the investigated list coloring problems in terms of the number of colors and the diameter of the input graph $G$, taking into account the results presented in the current article.
In Section~\ref{sec:3PreExt} we show our hardness result for \textsc{Retract to $C_6$} and use it to prove that \kPre{3} is $\NP$-complete on bipartite graphs with diameter four.
In Section~\ref{sec:C6Hom} we show that {\Comp} and {\Hom} are $\NP$-complete on bipartite graphs with diameter four.
In Section~\ref{sec:3biclique}, we point out the flaw in the hardness proof for $3$-\textsc{Biclique} on bipartite graphs presented by Fleischner et al.~\cite{FMPS.09}, and show how to obtain the result using the results presented in Section~\ref{sec:C6Hom}.
Finally, in Section~\ref{sec:3fall} we provide a reduction from $3$-\textsc{Fall Coloring} to \kPre{3} that preserves the diameter of the input graph, and prove that $3$-\textsc{Fall Coloring} is $\NP$-complete on bipartite graphs with diameter four.



\section{Definitions and notation}\label{sec:preliminaries}
All graphs that we consider are simple, that is, they do not have loops nor multiple edges.
We can also assume that the graphs are connected, as otherwise it suffices to solve the problems on the connected components.
For basic definitions on graph theory, we refer the reader to~\cite{W96}.

For $u,v \in V(G)$, we denote by ${\sf dist}(u,v)$ the length of a shortest path between $u$ $v$, and it is called the \emph{distance between $u$ and $v$ in $G$}.
The \emph{diameter} of a graph $G$ is the maximum distance between any pair of vertices of $G$. We denote by $\mathcal{B}$ the class of bipartite graphs. Also, for a fixed positive integer $d$, we denote by ${\cal B}_d$ the class of bipartite graphs with diameter at most $d$, and by $\mathcal{D}_d$ the class of graphs with diameter at most $d$.
A set $X \subseteq V(G)$ is a \emph{dominating set} of $G$ if every vertex of $G$ is either in $X$ or has a neighbor in $X$.
Similarly, we say that $A \subseteq V(G)$ \emph{dominates} $B \subseteq V(G)$ if every vertex of $B$ is either in $A$ or has a neighbor in $A$.
For a graph $H \subseteq G$, we denote by $N_H(v)$ the set of neighbors of $v$ that are contained in $V(H)$.

For an integer $\ell \geq 1$, we denote by $[\ell]$ the set $\{1, \ldots, \ell\}$. Given problems $P$ and $P'$, we write $P\preceq P'$ to denote that there exists a polynomial reduction from $P$ to $P'$ (hence, $P'$ is at least as hard as $P$). Also, given a graph class ${\cal G}$ and a problem $P$, we denote $P$ restricted to ${\cal G}$ by $P|_{{\cal G}}$.
Given sets $A$ and $B$, a function $f:A\rightarrow B$, and an element $b\in B$, the set of elements of $A$ whose image is $b$ is denoted by $f^{-1}(b)$. Also, given $X\subseteq A$, we denote the image of $X$ by $f(X)$, i.e., $f(X) = \{f(x)\mid x\in X\}$. If $f(X) = \{b\}$, we abuse notation and write $f(X) = b$.

Given a graph $G$ and a positive integer $k$, we say that a function $f:V(G)\rightarrow[k]$ is a \emph{proper $k$-coloring of $G$} if $f(u)\neq f(v)$ for every $uv\in E(G)$.
A vertex $v$ is called a \emph{b-vertex} if $f(N[v]) = [k]$, and $f$ is a \emph{$k$-fall-coloring of $G$} if every vertex of $G$ is a b-vertex.
A \emph{list assignment of $G$} is a function $L:V(G)\rightarrow 2^{\mathbb{N}}$ that assigns to each vertex $u \in V(G)$ a finite subset $L(u)$ of positive integers.
A \emph{partial $k$-coloring of $G$} is a function $p:V' \to [k]$ where $V' \subseteq V(G)$ and $p$ is a proper coloring of the subgraph of $G$ induced by $V'$. If $V'$ is not given, we denote it by ${\sf dom}(p)$. Given a partial $k$-coloring $p$ of $G$, we say that $f$ is a \emph{$k$-extension of $p$} if $f$ is a proper $k$-coloring of $G$ such that $f(u) = p(u)$ for every $u\in {\sf dom}(p)$.

A \emph{biclique} of a graph $G$ is a pair of vertex non-empty subsets $(A,B)$ such that $G' = (A\cup B, E')$  is a complete bipartite graph, where $G'$ is the subgraph of $G$ with vertex set $A \cup B$, and edge set $E'$ containing  the edges of $G$ between $A$ and $B$. A \emph{$k$-biclique partition} of $G$ is a set of $k$ disjoint bicliques $\{(A_1,B_1),\cdots,(A_k,B_k)\}$ such that $\bigcup_{i=1}^k (A_i\cup B_i) = V(G)$.

The coloring problems investigated in this article are formally defined below, adopting the notation from other papers; see e.g.~\cite{P.16}.
In each of the problems defined below, we consider $k$ to be a fixed integer with $k \geq 1$.

\vspace{1mm}
\begin{myproblem}
  \problemtitle{\textsc{$k$-List Coloring}}
  \probleminput{A graph $G = (V,E)$ and a list assignment $L$ s.t. $|L(u)|\le k$ for every $u\in V(G)$.}
  \problemquestion{Does $G$ admit a proper coloring $f$ s.t. $f(u)\in L(u)$ for every $u\in V(G)$?}
\end{myproblem}

\vspace{1mm}
\begin{myproblem}
  \problemtitle{\textsc{List $k$-Coloring}}
  \probleminput{A graph $G = (V,E)$ and a list assignment $L$ s.t. $L(u) \subseteq [k]$ for every $u\in V(G)$.}
  \problemquestion{Does $G$ admit a proper coloring $f$ s.t. $f(u)\in L(u)$ for every $u\in V(G)$?}
\end{myproblem}

\vspace{1mm}
\begin{myproblem}
  \problemtitle{\textsc{$k$-PreExt} }
  \probleminput{A graph $G = (V,E)$, a positive integer $k$, and a partial $k$-coloring $p$ of $G$.}
  \problemquestion{Does $p$ have a $k$-extension?}
\end{myproblem}

\vspace{1mm}
\begin{myproblem}
  \problemtitle{\textsc{$k$-Biclique Partition}}
  \probleminput{A graph $G = (V,E)$.}
  \problemquestion{Does $G$ have a $k$-biclique partition?}
\end{myproblem}

\vspace{1mm}
\begin{myproblem}
  \problemtitle{\textsc{$k$-Fall-Coloring}}
  \probleminput{A graph $G = (V,E)$.}
  \problemquestion{Does $G$ admit a $k$-fall-coloring?}
\end{myproblem}
\vspace{1mm}

The investigated homomorphism problems are listed below, where $H$ is considered to be a fixed graph whenever it is part of the name of a problem.

\vspace{1mm}
\begin{myproblem}
  \problemtitle{\textsc{Surjective $H$-Homomorphism}}
  \probleminput{A graph $G = (V,E)$.}
  \problemquestion{Does $G$ have a surjective homomorphism to $H$?}
\end{myproblem}

\begin{myproblem}
  \problemtitle{\textsc{Edge-Surjective $H$-Homomorphism}}
  \probleminput{A graph $G = (V,E)$.}
  \problemquestion{Does $G$ admit an edge-surjective $H$-homomorphism?}
\end{myproblem}

\begin{myproblem}
  \problemtitle{\textsc{Retract}}
  \probleminput{A graph $G = (V,E)$ and a subgraph $F \subseteq G$.}
  \problemquestion{Is there a retraction of $G$ to $F$?}
\end{myproblem}

\begin{myproblem}
  \problemtitle{\textsc{Retract to $H$}}
  \probleminput{A graph $G = (V,E)$ and a subgraph $F \subseteq G$ isomorphic to $H$.}
  \problemquestion{Is there a retraction of $G$ to $F$?}
\end{myproblem}

A \emph{3-uniform hypergraph} is a hypergraph such that each hyperedge has size exactly~3. A \emph{2-coloring of a hypergraph $G$} is a function $f:V(G)\rightarrow \{1,2\}$ such that $f(e) = \{1,2\}$ for every hyperedge $e\in E(G)$ (i.e., no hyperedge is monochromatic).
Lovász~\cite{L.73} showed that the problem of 2-coloring 3-uniform hypergraphs, formally defined below, is $\NP$-complete. This problem will be used in the reductions of Sections~\ref{sec:3PreExt} and~\ref{sec:3fall}.

\begin{myproblem}
  \problemtitle{\textsc{3-Uniform Hypergraph 2-Coloring} (abbreviated as \textsc{3-Uniform 2-Col)}}
  \probleminput{A 3-uniform hypergraph $G = (V,E)$.}
  \problemquestion{Does $G$ admit a 2-coloring?}
\end{myproblem}

From here on, we let $n$ denote the number of vertices of the input graph of the problem under consideration.

\section{List coloring vs. diameter}\label{sec:kvsdiam}

In this section, we investigate the complexity of \kPre{k}, \LkCol{k}, and \kLCol{k} on bipartite graphs with diameter $d$, for a fixed integer $d \geq 2$. 
We provide a complete  picture about the hardness of these problems, leaving as open cases only \kPre{3} and \LkCol{3} for $d=3$.

First, notice that there is a straightforward reduction from \kPre{k} to \LkCol{k}.
Indeed, for each $v \in V(G)$, if $v$ is precolored with color $p(v)$, then define $L(v) = \{p(v)\}$; otherwise, define $L(v) = [k]$.
Furthermore, \LkCol{k} is a particular case of \kLCol{k},  since each vertex in an instance of the former problem has a list assignment of size at most $k$.
From these two remarks, we get
\begin{equation}
\kPre{k} \preceq \LkCol{k} \preceq \kLCol{k}.\label{eq:complexityorder}
\end{equation}

Since the reductions discussed above do not change the input graph, we get that Equation~(\ref{eq:complexityorder}) holds when we restrict the problems to graphs in $\mathcal{B}_d$.
We remark that there are polynomial-time algorithms for the \textsc{$2$-List Coloring} problem by Lovász~\cite{L.73} and Vizing\cite{V.76}, and hence for \kPre{2} and \LkCol{2} as well. In what follows we focus on the complexity of these problems for $k\ge 3$.
Observe that if \kPre{k} is proved to be $\NP$-complete on bipartite graphs for some $k$, then the same holds for the other problems, and the next result shows that it also implies that $\kPre{(k+1)}|_{{\cal B}_3}$ is $\NP$-complete.

\begin{proposition}\label{prop:complexityOrder}
Let $k \geq 1$ be a fixed  integer.
Then $\kPre{k}|_{{\cal B}} \preceq \kPre{(k+1)}|_{{\cal B}_3}$.
\end{proposition}
\begin{proof}
Let $G\in {\cal B}$ with parts $X$ and $Y$, $p$ be a partial $k$-coloring of $G$, and $(G',p')$ be an instance of   $\kPre{(k+1)}|_{{\cal B}_3}$ obtained from $(G,p)$ as follows.
To obtain $G'$, add to $G$ two new vertices $x$ and $y$, all edges from $x$ to vertices in $Y$, and all edges from $y$ to vertices in $X$; therefore $G'$ is a bipartite graph with parts $X'= X\cup \{x\}$ and $Y'=Y\cup \{y\}$. Let $p'$ be obtained from $p$ by giving color $k+1$ to both $x$ and $y$.
Now, any extension of $p'$ defines an extension of $p$ and vice-versa, since color $k+1$ can only appear in the new vertices.
Let us argue about the diameter of $G'$.
Recall that we can assume that $G$ is connected, and therefore it has no isolated vertices.

Consider a pair $u,v\in V(G')$.
If they are within the same part, say $X'$, then either they are both adjacent to $y$, or $u=x$ in which case $(x,w,v)$ is a path, where $w$ is any neighbor of $v$ in $Y$ (recall that $G$ has no isolated vertices).
If $u \in X'$ and $v\in Y'$, then either $\{u,v\}\cap \{x,y\}=\emptyset$, in which case $(u,y,w,v)$ is a $(u,v)$-path for any neighbor $w$ of $v$ with $w \in X$; or $\{u,v\}=\{x,y\}$ and $(x,w,w',y)$ is a $(u,v)$-path for any edge $ww'\in E(G)$; or $|\{u,v\}\cap \{x,y\}|=1$, in which case $uv\in E(G')$. Thus, we get that $G'$ has indeed diameter at most~3.\end{proof}

Table~\ref{table:diam3} (resp. Table~\ref{table:diamk}) presents the complexity of the problems discussed in this section for $k=3$ (resp. for every fixed $k \geq 4$) restricted to bipartite graphs with diameter at most 2, 3, and 4. Let us explain how these tables are filled.
In Section~\ref{sec:3PreExt} we prove that $\kPre{3}|_{{\cal B}_4}$ is $\NP$-complete.
Note that, from Equation~(\ref{eq:complexityorder}) and the fact that ${\cal B}_d \subseteq {\cal B}_{d+1}$, we get that the forth row downwards in Table~\ref{table:diam3} is filled with $\NP$-completeness results.
Also, by Proposition~\ref{prop:complexityOrder}, we get that for every fixed $k \ge 4$, row~$3$ downwards in Table~\ref{table:diamk} is filled with $\NP$-completeness results.
When $G$ is a complete bipartite graph, the polynomiality of \LkCol{k} and \kPre{k} follows from  previous results by Couturier et al.~\cite{CGKP.15} and Hoàng et al.~\cite{HKLSS.10}.
In Section~\ref{sec:complBipPoly}, we briefly discuss those results and show a simple algorithm for \LkCol{k} that is better than both when the input graph is a complete bipartite graph.
Finally, Jansen and Scheffler~\cite{JS.97} proved that \kLCol{3} is $\NP$-complete on complete bipartite graphs.
The last two sentences justify the second row of both tables.

\begin{table}
\centering
\begin{tabular}[t]{|c|c|c|c|}
\hline
Diameter & \kPre{3} & \LkCol{3} & \kLCol{3} \\
\hline
2 & $\P$ & $\P$\ \cite{HKLSS.10,CGKP.15} & $\NPc$\ \cite{HJP.15} \\
\hline
3 & open & open & $\NPc$ \\
\hline
4 & $\NPc$\ (Sec. \ref{sec:3PreExt}) & $\NPc$ & $\NPc$ \\
\hline
\end{tabular}
\medskip
\caption{Row labeled $i$  presents the complexity of the corresponding problems restricted to bipartite graphs with diameter at most $i$.}\label{table:diam3}
\end{table}

\begin{table}
\begin{center}
\begin{tabular}[t]{|c|c|c|c|}
\hline
Diameter       & \kPre{k} & \LkCol{k} & \kLCol{k} \\
\hline
2 & $\P$ & $\P$\ \cite{HKLSS.10,CGKP.15} & $\NPc$ \\
\hline
3 & $\NPc$ {(Prop. \ref{prop:complexityOrder})} &  $\NPc$ & $\NPc$ \\
\hline
4 & $\NPc$ & $\NPc$ & $\NPc$ \\
\hline
\end{tabular}
\medskip
\caption{Row labeled $i$ presents the complexity of the corresponding problems restricted to bipartite graphs with diameter at most $i$, for every fixed integer $k \geq 4$.}\label{table:diamk}
\end{center}
\vspace{-.4cm}
\end{table}

\subsection{\textsc{List \texorpdfstring{$k$}{k}-Coloring} is polynomial on complete bipartite graphs}\label{sec:complBipPoly}

We begin this section by discussing some previously known results for \LkCol{k}.
For $r \geq 1$, let $rH$ denote the graph formed by $r$ disjoint copies of a graph $H$, and by $G + H$ we denote the graph formed by the disjoint union of the graphs $G$ and $H$.

Hoàng et al.~\cite{HKLSS.10} considered the \LkCol{k} problem in $P_5$-free graphs.
A key ingredient of their algorithm is a result by Bacsó and Tuza~\cite{Bacso.Tuza.90} stating that every $P_5$-free graph contains a \emph{dominating structure}, that is, an induced subgraph $D \subseteq G$ that is isomorphic to $P_3$ or to a complete graph, and $V(D)$ is a dominating set of $G$.
Their algorithm starts by performing a brute-force search for a dominating structure $D$ of size at most $k$ over all $\binom{n}{k}$ sets of vertices of size $k$ in $G$.
If none is found, then by the aforementioned result by Bacsó and Tuza~\cite{Bacso.Tuza.90}, $G$ contains a clique of size at least $k+1$ and thus it is not $k$-colorable.
Otherwise, they guess a coloring of $D$ and reduce the original \LkCol{k} instance to solving  at most $(kn)^{k^5}$ instances of \LkCol{(k-1)}, which are then solved recursively.

Couturier et al.~\cite{CGKP.15} generalized the result by Hoàng et al.~\cite{HKLSS.10} to $(rP_1 + P_5)$-free graphs.
They showed that any such graph that is $L$-colorable, where $L$ is a list assignment of $G$ with $L(u) \subseteq [k]$ for every $u \in V(G)$, admits another kind of dominating set of size at most $f(k)$, for some computable function $f$, and their algorithm for \LkCol{k} also depends on finding a dominating set of size at most $f(k)$ in $G$  and on solving $O(n^{f(k)})$ new instances.

We provide a simple algorithm for \LkCol{k} in complete bipartite graphs that has the advantage of being an {$\FPT$} algorithm, while the algorithms discussed above are clearly {\XP}.
 As mentioned before, this algorithm also solves \kPre{k} on complete bipartite graphs, by Equation~(\ref{eq:complexityorder}) and the fact that the reduction does not modify the input graph.


\begin{theorem}\label{thm:FPT-kCOl-diam2}
The $\LkCol{k}|_{{\cal B}_2}$ problem can be solved in time $\mathcal{O}(2^k \cdot k\cdot\log k \cdot n)$. In particular, it can be solved in linear time for every fixed integer $k \geq 1$.
\end{theorem}
\begin{proof}
Consider a complete bipartite graph $G=(A\cup B, E)$ and a list assignment $L$ such that $L(u)\subseteq [k]$ for every $u\in V(G)$. It is easy to see that solving $\LkCol{k}$ on $G$ is equivalent to choosing an element $c(u)$ in $L(u)$ for each $u\in V(G)$ in such a way that the chosen elements for $A$ have no intersection with the chosen elements of $B$; more formally, $\{c(u)\mid u\in A\}\cap \{c(u)\mid u\in B\} = \emptyset$. Because choosing a color $c(u)$ for $u\in A$ actually makes this color available for every $v\in A$ and non-available for every $v\in B$, the problem actually consists of choosing which subset of colors $S$ can be used in $A$, knowing that $\overline{S} = [k]\setminus S$ will be the colors available for $B$. So, we can simply test, for each subset $S\subseteq [k]$, whether $S\cap L(u)\neq \emptyset$ for every $u\in A$, and $\overline{S}\cap L(u)\neq \emptyset$ for every $u\in B$. This takes time $\mathcal{O}(2^k \cdot k\cdot\log k \cdot n)$ since we need to check the intersection between either $S$ or $\overline{S}$ and $L(u)$ (which takes time $O(k\cdot \log k)$ if we consider the lists are ordered) for every $u\in V(G)$.
\end{proof}

In fact, Theorem~\ref{thm:FPT-kCOl-diam2} states that $\LkCol{k}|_{{\cal B}_2}$ is \emph{fixed-parameter tractable} parameterized by $k$, using terminology from parameterized complexity (cf. for instance~\cite{CyganFKLMPPS15}). On the other hand, by considering the disjoint union of instances of $\LkCol{k}|_{{\cal B}_2}$ it is easy to obtain a so-called \emph{\textsc{and}-cross-composition} (see again~\cite{CyganFKLMPPS15}), hence refuting, under standard complexity assumptions, the existence of \emph{polynomial kernels} for the problem $\LkCol{k}$, parameterized by $k$, restricted to the graph class consisting of the disjoint union of complete bipartite graphs. Note that $\LkCol{k}$ restricted to this class is also  fixed-parameter tractable, since one can solve the problem independently on each connected component. We leave as an open problem the existence of a polynomial kernel for  $\LkCol{k}|_{{\cal B}_2}$ parameterized by $k$.


\section{\texorpdfstring{\RetF{C_6}}{\textsc{Retract to C{\small6}}}}
\label{sec:3PreExt}

Our goal is to show that \RetF{C_6} is $\NP$-complete even under various constraints, which imply that the input graph is in ${\cal B}_4$.
These constraints shall be particularly useful in the next section, where we prove the $\NP$-completeness of {\Comp} and {\Hom} on ${\cal B}_4$.
The proof of the following theorem consists of an appropriate modification of a reduction by Kratochvíl~\cite{K.93}. Basically, we start from the {\sc 3-Uniform 2-Col} problem, instead of {\sc 1-in-3 SAT}, which allows us to combine a small gadget introduced by Kratochvíl (called $AB$-link) in a unique edge-gadget (instead, Kratochvíl had to combine them in eight distinct ways, depending on whether each variable was positive or negative). Additionally, since we do not need planarity, we can start with only six pre-colored vertices, which allows us to bound the diameter of $G$.

\begin{theorem}\label{theo:C6retract}
Let $G = (X\cup Y)$ be a bipartite graph, let $C\subseteq G$ be an induced $C_6$ in $G$, and let $Y_C=V(C)\cap Y$. Deciding whether $G$ has a retraction to $C$ is $\NP$-complete, even if $Y_C$ dominates $X$ and ${\sf dist}(h,y)\le 2$ for every $h\in Y_C$ and $y\in Y$.
\end{theorem}
\begin{proof}
We reduce from the {\sc 3-Uniform 2-Col} problem (recall the definition from Section~\ref{sec:preliminaries}). For this, consider a 3-uniform hypergraph $H=(V,E)$, and let $G$ be the bipartite graph with bipartition $(V,E)$ such that $ue\in E(G)$ if and only if $u\in e$. Add vertices $p^V_1,p^V_2,p^V_3$ and $p^E_1,p^E_2,p^E_3$ to parts $V$ and $E$, respectively, and make the subgraph induced by these vertices be the cycle $C = (p^V_1,p^E_2,p^V_3,p^E_1,p^V_2,p^E_3,p^V_1)$.
Add an edge between each $v\in V$ and $p^E_3$. This ensures that any retraction $f$ from $G$ to $C$ is such that $f(V) = \{p^V_1,p^V_2\} = N_C(p_3^E)$.
Now, for each hyperedge $e\in E$, we replace some of the edges incident to $e$ with an edge gadget defined as follows. For easier reference, let $V=\{v_1,\ldots,v_n\}$ and $E = \{e_1,\ldots,e_m\}$. Let $e_j\in E$, and let $i_1,i_2,i_3$ be the indices of the vertices within $e_j$.
Remove edges $v_{i_1}e_j$ and $v_{i_2}e_j$ from $G$, and replace them with the gadget of~Figure~\ref{fig:edgegadget}. For better visibility, sometimes we make more than one copy of some of the vertices of $C$ in the figure (for instance, vertex $p^E_3$ is represented 3 different times, but all occurrences correspond to the same vertex).  Dashed and solid vertices are used to represent the bipartition of $G$. The purpose of this gadget is to ensure that hyperedge $e_j$ cannot be monochromatic, as discussed below.

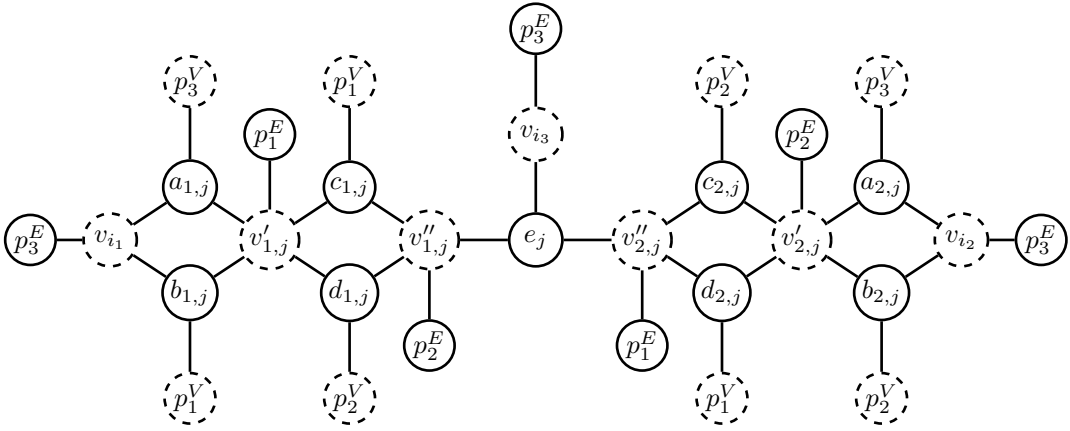
\begin{figure}[thb]
\begin{center}
  \begin{tikzpicture}[scale=.7]
  \pgfsetlinewidth{1pt}
  \tikzset{vertex/.style={circle, minimum size=0.7cm, draw, inner sep=1pt}}
  \tikzset{simple/.style={circle, minimum size=0.3cm, draw, inner sep=1pt}}

  \node [vertex] (ej) at (0,0){$e_j$};
  \node [vertex,dashed] (vpp1j) at (-2,0){$v''_{1,j}$};
  \node [vertex,dashed] (vp1j) at (-5,0){$v'_{1,j}$};
  \node [vertex,dashed] (v1j) at (-8,0){$v_{i_1}$};

  \node [simple] (d1j) at (-3.5,-1){$d_{1,j}$};
  \node [simple] (c1j) at (-3.5,1){$c_{1,j}$};
  \node [simple] (b1j) at (-6.5,-1){$b_{1,j}$};
  \node [simple] (a1j) at (-6.5,1){$a_{1,j}$};

  \node [simple] (p1E2) at (-2,-2){$p^E_2$};
  \node [simple,dashed] (p1V2) at (-3.5,-3){$p^V_2$};
  \node [simple,dashed] (p1V1) at (-6.5,-3){$p^V_1$};

  \node [simple,dashed] (p1V12) at (-3.5,3){$p^V_1$};
  \node [simple] (p1E1) at (-5,2){$p^E_1$};
  \node [simple,dashed] (p1V3) at (-6.5,3){$p^V_3$};

  \node [simple] (p1E3) at (-9.5,0){$p^E_3$};

  \draw (ej)--(vpp1j)--(p1E2) (vpp1j)--(d1j)--(p1V2) (vpp1j)--(c1j)--(p1V12) (c1j)--(vp1j)--(d1j) (vp1j)--(p1E1) (vp1j)--(b1j)--(v1j)--(p1E3) (vp1j)--(a1j)--(v1j) (b1j)--(p1V1) (a1j)--(p1V3);

  \node [vertex,dashed] (vpp2j) at (2,0){$v''_{2,j}$};
  \node [vertex,dashed] (vp2j) at (5,0){$v'_{2,j}$};
  \node [vertex,dashed] (v2j) at (8,0){$v_{i_2}$};

  \node [simple] (d2j) at (3.5,-1){$d_{2,j}$};
  \node [simple] (c2j) at (3.5,1){$c_{2,j}$};
  \node [simple] (b2j) at (6.5,-1){$b_{2,j}$};
  \node [simple] (a2j) at (6.5,1){$a_{2,j}$};

  \node [simple] (p2E1) at (2,-2){$p^E_1$};
  \node [simple,dashed] (p2V1) at (3.5,-3){$p^V_1$};
  \node [simple,dashed] (p2V2) at (6.5,-3){$p^V_2$};

  \node [simple,dashed] (p2V22) at (3.5,3){$p^V_2$};
  \node [simple] (p2E2) at (5,2){$p^E_2$};
  \node [simple,dashed] (p2V3) at (6.5,3){$p^V_3$};

  \node [simple] (p2E3) at (9.5,0){$p^E_3$};

  \draw (ej)--(vpp2j)--(p2E1) (vpp2j)--(d2j)--(p2V1) (vpp2j)--(c2j)--(p2V22) (c2j)--(vp2j)--(d2j) (vp2j)--(p2E2) (vp2j)--(b2j)--(v2j)--(p2E3) (vp2j)--(a2j)--(v2j) (b2j)--(p2V2) (a2j)--(p2V3);

  \node [vertex,dashed] (v3j) at (0,2){$v_{i_3}$};
  \node [simple] (p3E3) at (0,4){$p^E_3$};

  \draw (ej)--(v3j)--(p3E3);

\end{tikzpicture}
\caption{Gadget for hyperedge $e_j = \{v_{i_1},v_{i_2},v_{i_3}\}$.}
\label{fig:edgegadget}
\end{center}
\end{figure}

We need to prove that $G$ is a bipartite graph, that $C$ has the desired properties, and that $f$ has a retraction to $C$ if and only if $H$ has an appropriate 2-coloring. We first prove the latter. So, let  $f$ be a retraction to $C$; we prove that $f$ restricted to $V$ is a 2-coloring, using colors $p^V_1,p^V_2$, such that no hyperedge is monochromatic. Suppose otherwise and let $e_j = \{v_{i_1},v_{i_2},v_{i_3}\}$ be such that $f(\{v_{i_1},v_{i_2},v_{i_3}\})$ is monochromatic. Because every vertex in $V$ is adjacent to $p^E_3$, we get that $f(\{v_{i_1},v_{i_2},v_{i_3}\})\in \{p^V_1,p^V_2\} = N_C(p^E_3)$. First, suppose that $f(\{v_{i_1},v_{i_2},v_{i_3}\}) = p^V_1$. We prove that $f$ must be as depicted in Figure~\ref{fig:coloring1}, a contradiction since in this case $e_j$ has neighbors in $f^{-1}(p^E_i)$  for every $i\in [3]$, and therefore cannot be mapped to~$C$. Labels of vertices that do not have their image forced are left blank.
Note that $C$ can also be seen as the complete bipartite graph minus the perfect matching $\{p^V_ip^E_i\mid i\in[3]\}$.
Because $a_{1,j}$ is adjacent to $p^V_3$ and to $v_{i_1}\in f^{-1}(p^V_1)$, we get that $f(a_{1,j}) = p^E_2$. But now we get that $v'_{1,j}$ is adjacent to $p_1^E$ and to $f^{-1}(p^E_2)$, and therefore must be in $p^V_3$. This implies that $d_{1,j}$ is adjacent to $p^V_2$ and $f^{-1}(p^V_3)$, and hence is colored with $p^E_1$. Finally, we get that $v''_{1,j}$ is adjacent to $p^E_2$ and $f^{-1}(p^E_1)$, and must be colored with $p^V_3$. The analysis for the right-hand side of the gadget is similar. For the case where $f(\{v_{i_1},v_{i_2},v_{i_3}\}) = p^V_2$, we have the situation depicted in Figure~\ref{fig:coloring2}, and it follows similarly. Therefore, no edge is monochromatic, as we wanted to prove.

\begin{figure}[thb]
\begin{center}
  \begin{tikzpicture}[scale=.65]
  \pgfsetlinewidth{1pt}
  \tikzset{vertex/.style={circle, minimum size=0.7cm, draw, inner sep=1pt}}
  \tikzset{simple/.style={circle, minimum size=0.3cm, draw, inner sep=1pt}}

  \node [vertex] (ej) at (0,0){$e_j$};
  \node [vertex,label={[xshift=-.2cm]-30:$v''_{1,j}$}] (vpp1j) at (-2,0){$p^V_3$};
  \node [vertex,label = -90:$v'_{1,j}$] (vp1j) at (-5,0){$p^V_3$};
  \node [vertex,label=-90:$v_{i_1}$,line width=2pt] (v1j) at (-8,0){$p^V_1$};

  \node [vertex,label={[yshift=+.2cm,xshift=-0.2cm]-30:$d_{1,j}$}] (d1j) at (-3.5,-1){$p^E_1$};
  \node [simple] (c1j) at (-3.5,1){};
  \node [simple] (b1j) at (-6.5,-1){};
  \node [vertex,label={[xshift=+.1cm]120:$a_{1,j}$}] (a1j) at (-6.5,1){$p^E_2$};

  \node [vertex] (p1E2) at (-2,2){$p^E_2$};
  \node [vertex] (p1V2) at (-3.5,-3){$p^V_2$};
  \node [vertex] (p1V1) at (-6.5,-3){$p^V_1$};

  \node [vertex] (p1V12) at (-3.5,3){$p^V_1$};
  \node [vertex] (p1E1) at (-5,2){$p^E_1$};
  \node [vertex] (p1V3) at (-6.5,3){$p^V_3$};

  \node [vertex] (p1E3) at (-9.5,0){$p^E_3$};

  \draw (ej)--(vpp1j)--(p1E2) (vpp1j)--(d1j)--(p1V2) (vpp1j)--(c1j)--(p1V12) (c1j)--(vp1j)--(d1j) (vp1j)--(p1E1) (vp1j)--(b1j)--(v1j)--(p1E3) (vp1j)--(a1j)--(v1j) (b1j)--(p1V1) (a1j)--(p1V3);

  \node [vertex,label={[xshift=+0.2cm]-120:$v''_{2,j}$}] (vpp2j) at (2,0){$p^V_2$};
  \node [vertex,label=-90:$v'_{2,j}$] (vp2j) at (5,0){$p^V_1$};
  \node [vertex,label=90:$v_{i_2}$,line width=2pt] (v2j) at (8,0){$p^V_1$};

  \node [simple] (d2j) at (3.5,-1){};
  \node [vertex,label=150:$c_{2,j}$] (c2j) at (3.5,1){$p^E_3$};
  \node [vertex,label=-30:$b_{2,j}$] (b2j) at (6.5,-1){$p^E_3$};
  \node [simple] (a2j) at (6.5,1){};

  \node [vertex] (p2E1) at (2,-2.5){$p^E_1$};
  \node [vertex] (p2V1) at (3.5,-3){$p^V_1$};
  \node [vertex] (p2V2) at (6.5,-3){$p^V_2$};

  \node [vertex] (p2V22) at (3.5,3){$p^V_2$};
  \node [vertex] (p2E2) at (5,2){$p^E_2$};
  \node [vertex] (p2V3) at (6.5,3){$p^V_3$};

  \node [vertex] (p2E3) at (9.5,0){$p^E_3$};

  \draw (ej)--(vpp2j)--(p2E1) (vpp2j)--(d2j)--(p2V1) (vpp2j)--(c2j)--(p2V22) (c2j)--(vp2j)--(d2j) (vp2j)--(p2E2) (vp2j)--(b2j)--(v2j)--(p2E3) (vp2j)--(a2j)--(v2j) (b2j)--(p2V2) (a2j)--(p2V3);

  \node [vertex,label=30:$v_{i_3}$,line width=2pt] (v3j) at (0,2){$p^V_1$};
  \node [vertex] (p3E3) at (0,4){$p^E_3$};

  \draw (ej)--(v3j)--(p3E3);
\end{tikzpicture}
\caption{Coloring of a hyperedge gadget when $f(\{v_{i_1},v_{i_2},v_{i_3}\}) = p^V_1$. Vertices $v_{i_1}$, $v_{i_2}$, and $v_{i_3}$ are emphasized. The image of a vertex is put inside of the node, and its label appears next to it, with the exception of $e_j$. Note that the blank vertices can be colored, from left to right, with $p^E_2$, $p^E_2$, $p^E_3$, and $p^E_2$.}
\label{fig:coloring1}
\end{center}
\end{figure}
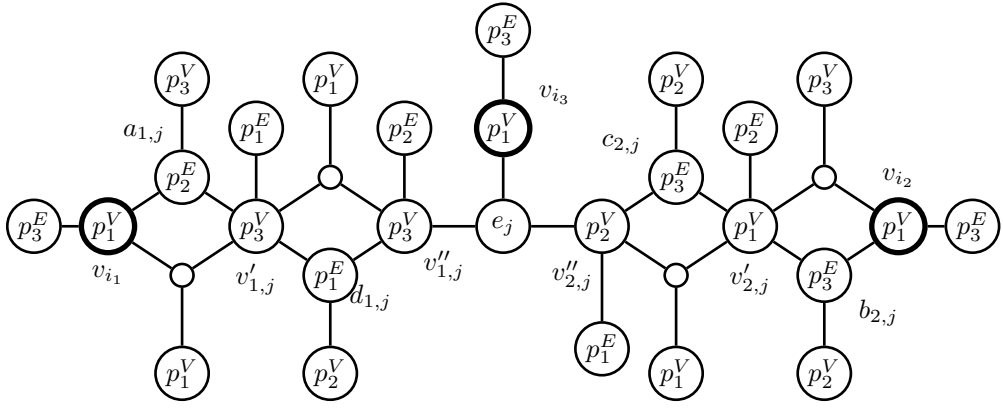

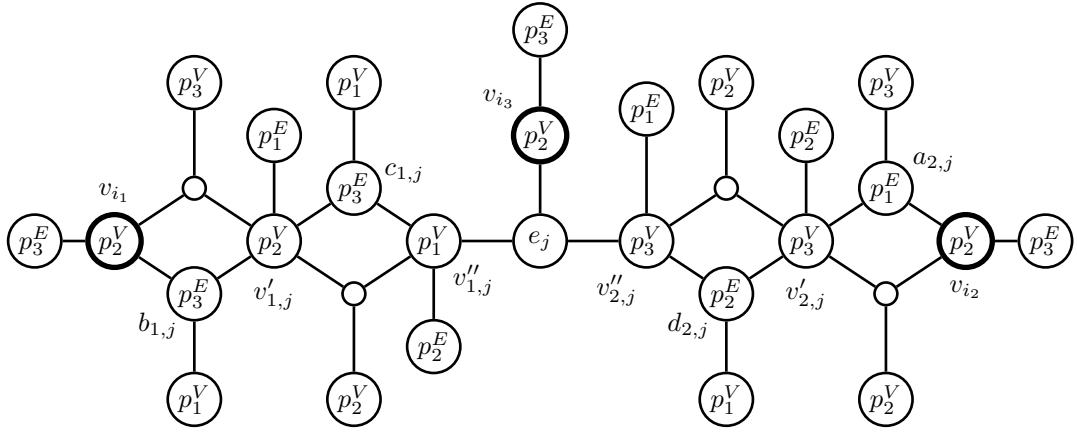
\begin{figure}[thb]
\begin{center}
  \begin{tikzpicture}[scale=.7]
  \pgfsetlinewidth{1pt}
  \tikzset{vertex/.style={circle, minimum size=0.7cm, draw, inner sep=1pt}}
  \tikzset{simple/.style={circle, minimum size=0.3cm, draw, inner sep=1pt}}

  \node [vertex] (ej) at (0,0){$e_j$};
  \node [vertex,label={[xshift=-.2cm]-30:$v''_{1,j}$}] (vpp1j) at (-2,0){$p^V_1$};
  \node [vertex,label = -90:$v'_{1,j}$] (vp1j) at (-5,0){$p^V_2$};
  \node [vertex,label=90:$v_{i_1}$,line width=2pt] (v1j) at (-8,0){$p^V_2$};

  \node [simple] (d1j) at (-3.5,-1){};
  \node [vertex,label={[yshift=+.2cm,xshift=-0.1cm]0:$c_{1,j}$}] (c1j) at (-3.5,1){$p^E_3$};
  \node [vertex,label={[xshift=+.1cm,yshift=+.2cm]-120:$b_{1,j}$}] (b1j) at (-6.5,-1){$p^E_3$};
  \node [simple] (a1j) at (-6.5,1){};

  \node [vertex] (p1E2) at (-2,-2){$p^E_2$};
  \node [vertex] (p1V2) at (-3.5,-3){$p^V_2$};
  \node [vertex] (p1V1) at (-6.5,-3){$p^V_1$};

  \node [vertex] (p1V12) at (-3.5,3){$p^V_1$};
  \node [vertex] (p1E1) at (-5,2){$p^E_1$};
  \node [vertex] (p1V3) at (-6.5,3){$p^V_3$};

  \node [vertex] (p1E3) at (-9.5,0){$p^E_3$};

  \draw (ej)--(vpp1j)--(p1E2) (vpp1j)--(d1j)--(p1V2) (vpp1j)--(c1j)--(p1V12) (c1j)--(vp1j)--(d1j) (vp1j)--(p1E1) (vp1j)--(b1j)--(v1j)--(p1E3) (vp1j)--(a1j)--(v1j) (b1j)--(p1V1) (a1j)--(p1V3);

  \node [vertex,label={[xshift=+0.2cm]-120:$v''_{2,j}$}] (vpp2j) at (2,0){$p^V_3$};
  \node [vertex,label=-90:$v'_{2,j}$] (vp2j) at (5,0){$p^V_3$};
  \node [vertex,label=-90:$v_{i_2}$,line width=2pt] (v2j) at (8,0){$p^V_2$};

  \node [vertex,label={[xshift=+.1cm,yshift=+.2cm]-120:$d_{2,j}$}] (d2j) at (3.5,-1){$p^E_2$};
  \node [simple] (c2j) at (3.5,1){};
  \node [simple] (b2j) at (6.5,-1){};
  \node [vertex,label={[xshift=-.1cm,yshift=-.1cm]30:$a_{2,j}$}] (a2j) at (6.5,1){$p^E_1$};

  \node [vertex] (p2E1) at (2,2.5){$p^E_1$};
  \node [vertex] (p2V1) at (3.5,-3){$p^V_1$};
  \node [vertex] (p2V2) at (6.5,-3){$p^V_2$};

  \node [vertex] (p2V22) at (3.5,3){$p^V_2$};
  \node [vertex] (p2E2) at (5,2){$p^E_2$};
  \node [vertex] (p2V3) at (6.5,3){$p^V_3$};

  \node [vertex] (p2E3) at (9.5,0){$p^E_3$};

  \draw (ej)--(vpp2j)--(p2E1) (vpp2j)--(d2j)--(p2V1) (vpp2j)--(c2j)--(p2V22) (c2j)--(vp2j)--(d2j) (vp2j)--(p2E2) (vp2j)--(b2j)--(v2j)--(p2E3) (vp2j)--(a2j)--(v2j) (b2j)--(p2V2) (a2j)--(p2V3);

  \node [vertex,label={[yshift=-.1cm]120:$v_{i_3}$},line width=2pt] (v3j) at (0,2){$p^V_2$};
  \node [vertex] (p3E3) at (0,4){$p^E_3$};

  \draw (ej)--(v3j)--(p3E3);
\end{tikzpicture}
\caption{Coloring of a hyperedge gadget when $f(\{v_{i_1},v_{i_2},v_{i_3}\}) = p^V_2$. Vertices $v_{i_1}$, $v_{i_2}$, and $v_{i_3}$ are emphasized. The image of a vertex is put inside of the node, and its label appears next to it, with the exception of $e_j$. Note that the blank vertices can be colored, from left to right, with $p^E_1$, $p^E_3$, $p^E_1$, and $p^E_1$.}
\label{fig:coloring2}
\end{center}
\end{figure}

Conversely, suppose now that $f'$ is a 2-coloring of $H$ with no monochromatic hyperedge, and let $f$ be obtained from $f'$ by mapping to $p^V_i$ all the vertices colored with~$i$, for $i\in \{1,2\}$. Let $e_j = \{v_{i_1},v_{i_2},v_{i_3}\}$. We show how to map the vertices within the hyperedge gadget related to $e_j$. The possibilities are the following:

\begin{itemize}
\item[$\bullet$] $f(v_{i_1}) = f(v_{i_2}) = p^V_1$: in this case $f(v_{i_3})=p^V_2$ since $e_j$ is not monochromatic. Map the vertices as in Figure~\ref{fig:coloring1}, except for $v_{i_3}$, and note that we can map $e_j$ to $p^E_1$, and that the blank vertices can be mapped to $C$;

\item[$\bullet$] $f(v_{i_1}) = f(v_{i_2}) = p^V_2$: in this case $f(v_{i_3})=p^V_1$ since $e_j$ is not monochromatic. Map the vertices as in Figure~\ref{fig:coloring2}, except for $v_{i_3}$, and note that $e_j$ can be mapped to $p^V_2$, and that the blank vertices can be mapped to $C$;

\item[$\bullet$]  $f(v_{i_1}) = p^V_1$ and $f(v_{i_2}) = p^V_2$: map the left-hand side as in Figure~\ref{fig:coloring1}, and the right-hand side as in Figure~\ref{fig:coloring2}. Note that $e_j$ can be mapped to $p\in \{p^E_1,p^E_2\}\setminus\{f(v_{i_3})\}$, and that the blank vertices can be mapped to $C$;

\item[$\bullet$]  $f(v_{i_1}) = p^V_2$ and $f(v_{i_2}) = p^V_1$: map the left-hand side as in Figure~\ref{fig:coloring2}, and the right-hand side as in Figure~\ref{fig:coloring1}. Note that $e_j$ can be mapped to $p^E_3$, and that the blank vertices can be mapped to $C$.
\end{itemize}

Finally, we need to prove that $G$ is bipartite, and that $C$ has the desired properties. To see that $G$ is bipartite, just observe that the dashed and solid vertices in Figure~\ref{fig:edgegadget} form a bipartition of $G$; let $(X,Y)$ be such a bipartition, where $V\subseteq X$ and $E\subseteq Y$ (i.e., $X$ contains the dashed vertices, and $Y$ the solid ones). More formally, one can see that
\[X = V\cup \{p^V_1,p^V_2,p^V_3\}\cup \{v'_{i,j},v''_{i,j}\mid i\in [2],j\in[m]\};\]
\[Y= E \cup \{p^E_1,p^E_2,p^E_3\}\cup \{a_{i,j},b_{i,j},c_{i,j},d_{i,j}\mid i\in [2],j\in[m]\}.\]

As for the properties of $C$, observe first that every $x\in X$ is adjacent to $p^E_i$, for some $i\in \{1,2,3\}$, i.e., $Y_C = \{p^E_1,p^E_2,p^E_3\}$ dominates $X$. It remains to prove that every $p\in Y_C$ is at distance at most~2 from every $y\in Y$. Given $e_j\in E$, denote by $Y_j$ the subset of vertices in $Y$ contained in a gadget related to $e_j$; hence $Y = Y_C\cup \bigcup_{j\in [m]} Y_j$ and it suffices to prove that, given some $j\in [m]$, every $p\in Y_C$ is at distance at most~2 from every vertex in $Y_j$. For $p^E_1$, it follows from the fact that $Y_j\subseteq N(\{v'_{1,j},v''_{2,j},p^V_2,p^V_3\})$ and that $\{v'_{1,j},v''_{2,j},p^V_2,p^V_3\}\subseteq N(p^E_1)$. For $p^E_2$, if follows from the fact that $Y_j\subseteq N(\{v'_{2,j},v''_{1,j},p^V_1,p^V_3\})$ and that $\{v'_{2,j},v''_{1,j},p^V_1,p^V_3\}\subseteq N(p^E_2)$. Finally, for $p^E_3$ if follows from the fact that $Y_j\subseteq N(\{v_{i_1},v_{i_2},v_{i_3},p^V_1,p^V_2\})$ and that $\{v_{i_1},v_{i_2},v_{i_3},p^V_1,p^V_2\}\subseteq N(p^E_3)$.
\end{proof}

\begin{corollary}\label{cor:3PreExt}
$\kPre{3}|_{{\cal B}_4}$ is $\NP$-complete, even if every vertex in one of the parts is adjacent to some precolored vertex.
\end{corollary}
\begin{proof}
Given a bipartite graph $G = (X\cup Y, E)$ and a $C_6$, say $C$, in $G$ with the properties stated in Theorem~\ref{theo:C6retract}, it suffices to precolor $p^E_i,p^V_i$ with $i$ for each $i\in [3]$. One can see that $G$ has a 3-extension for this precoloring if and only if $G$ has a retraction to $C$. Also, the fact that $\{p^E_1,p^E_2,p^E_3\}$ dominates $X$ gives us the property claimed in the statement. It remains to verify that ${\sf diam}(G)\le 4$. For this, first we argue that ${\sf dist}(x,y)\le 3$ for every $x\in X$ and $y\in Y$; indeed, either $xy\in E(G)$ or $(x,p^E_i,w,y)$ is a path between $x$ and $y$, where $p^E_i$ is any vertex in $N(x)\cap C$ and $(p^E_i,w,y)$ is the $(p^E_i,y)$-path of length~2 ensured in the statement of Theorem~\ref{theo:C6retract}. Now, for $x,x'\in X$, let $y\in N(x)$ (it exists by construction); we know from the previous sentence that ${\sf dist}(y,x')\le 3$ and thus it follows that ${\sf dist}(x,x')\le 4$. The argument for $y,y'\in Y$ is symmetric.
\end{proof}

\section{Surjective \texorpdfstring{$C_6$}{C6}-homomorphism}\label{sec:C6Hom}

We first prove that \Comp\ is $\NP$-complete on ${\cal B}_4$, and as a byproduct we get that \textsc{Surjective $C_6$-homomorphism} is $\NP$-complete (cf. Corollary~\ref{cor:Comp_Hom}).
Notice that, unlike in the retraction problem, in the {\Comp} problem, the target $C_6$ is not necessarily a fixed subgraph of $G$.
How\-ever, we show that, under some assumptions, an edge surjective $C_6$-homomorphism for $G$ coincides with a retraction of $G$ to $H$ for some choice of $H \subseteq G$ such that $H\cong C_6$.
We remark that another proof that \Comp\ is $\NP$-complete was given by Vikas~\cite{V.17}.
However, our proof is simpler and we consider constraints which are not present in the proof of Vikas~\cite{V.17}.

In what follows, we present a reduction from \RetF{C_6} on bipartite graphs of diameter four to \Comp.
Let $G$ and $H$ be the input graph and subgraph, respectively; also, let us write $H$ as $(h_1,\ldots, h_6)$, and $X,Y$ be the parts of $G$, with $\{h_1,h_3,h_5\}\subseteq Y$.
We first introduce the gadget depicted in Figure~\ref{fig:h1h2gadget} related to a vertex $u\in Y \setminus V(H)$. The cycle $(h'_1,\ldots,h'_6)$ represents a circular permutation of $(h_1,\ldots,h_6)$
such that $h'_1\in \{h_1,h_3,h_5\}$; we call this the \emph{$(h'_1,h'_4)$-gadget} (this is because $h'_1h'_4$ is the diagonal related to this gadget). We obtain the input graph $G'$ of \textsc{Surjective $C_6$-homomorphism} from $G$ by adding an $(h'_1,h'_4)$-gadget related to each $u \in Y \setminus V(H)$, for each possible pair $(h'_1,h'_4)$ such that $h'_1\in \{h_1,h_3,h_5\}$, namely for $(h_1,h_4)$, $(h_3,h_6)$, and $(h_5,h_2)$ (see Figure~\ref{fig:Diagonals}).
Note that $V(G') \setminus V(G)$ consists of precisely the vertices $\{a,b,c,d,e,g\}$ introduced for each choice of $u \in Y \setminus V(H)$ and each choice of $h'_1$, for a total of $6(|Y| - 3)$ new vertices.
We use the circular permutation in the following lemma to avoid making analogous arguments for each type of gadget separately.

\begin{figure}[ht]
\centering
\scalebox{.9}{
  \begin{tikzpicture}[scale=.8]
  \pgfsetlinewidth{1pt}
  \tikzset{vertex/.style={circle, minimum size=0.2cm, fill=black, draw, inner sep=1pt}}
  \tikzset{wvertex/.style={circle, minimum size=0.2cm, fill=white, draw, inner sep=1pt}}

\foreach \i in {1,3,5}{
\node[vertex,label=210-60*\i:{$h'_\i$}] (h\i) at (-\i*60+210:3.5) {};
}
\foreach \i in {2,4,6}{
\node[wvertex,label=210-60*\i:{$h'_\i$}] (h\i) at (-\i*60+210:3.5) {};
}
\node[wvertex,label=below:{$a$}] (a1) at (165:2.2) {};
\node[wvertex,label=above:{$b$}] (a2) at (135:2.2) {};
\node[wvertex,label=above:{$c$}] (c) at (30:2.2) {};
\node[vertex,label={[above,xshift=.4cm] $d$}] (d1) at (-15:2.1) {};
\node[vertex,label=270:{$e$}] (d2) at (-45:2.2) {};
\node[wvertex,label=left:{$g$}] (e) at (-90:2.2) {};

\node[vertex,label=right:{\large$u$}] (v) at (0,0) {};

 \draw (h1)--(h2)--(h3)--(h4)--(h5)--(h6)--(h1)
        (v)--(a1) (v)--(a2) (v)--(e) (v)--(c)
        (h1)--(a1) (h1)--(a2)--(d1) (a1)--(d2)--(e)
        (h4)--(d1)--(c)--(h3) (h4)--(d2) (h5)--(e);

\end{tikzpicture}
}
\caption{Gadget related to the diagonal $h'_1h'_4$.}
\label{fig:h1h2gadget}
\end{figure}
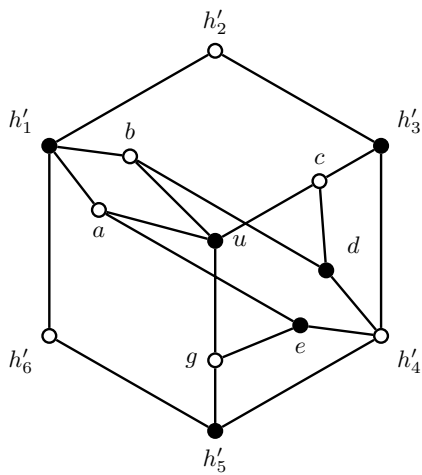

\begin{lemma}\label{lemma:diagonal}
    Let $G\in {\cal B}_4$ and $H\subseteq G$ be isomorphic to a $C_6$; write $H=(h_1,\ldots,h_6)$. Also, let $G'$ be obtained from $G$ by adding the gadget represented in Figure~\ref{fig:h1h2gadget} for every $u\in X$, and every $(h'_1,h'_4)\in \{(h_1,h_4),(h_3,h_6),(h_5,h_2)\}$. Let $f$ be an edge-surjective $C_6$-homomorphism of $G'$ to the $C_6$ $(1,\ldots,6)$. Suppose that $(h'_1,\ldots,h'_6)$ is a circular permutation of $H$ such that $h'_1\in \{h_1,h_3,h_5\}$, and $f(h'_i)=i$ for each $i\in\{3,4,5\}$. If some $u\in V(G)\setminus V(H)$ is such that $f(u) = 1$, then $f$ is a retraction of $G'$ to $H$.
\end{lemma}

\begin{proof}
We prove that if there exists a vertex $u\in V(G) \setminus V(H)$ such that $f(u)=1$, then $f(h'_1)=1$; note that the lemma follows since $\{1,3\} = \{f(h'_1),f(h'_3)\}\subseteq N(f(h'_2))$, and $\{1,5\} = \{f(h'_1),f(h'_5)\}\subseteq N(f(h'_2))$ imply that $f(h'_2) = 2$ and $f(h'_6) = 6$. To see that $f(h'_1)=1$, note that the images of the gadget related to $u$ must be as depicted in Figure~\ref{fig:lemma1}, as explained next (fixed values appear inside the vertex, while implied values appear between parenthesis next to the vertex). Since $g$ is adjacent to $u\in f^{-1}(1)$ and $h'_5\in f^{-1}(5)$, we get that $f(g) = 6$. Similarly, we get $f(c) = 2$ since it is adjacent to $u\in f^{-1}(1)$ and $h'_3\in f^{-1}(3)$. Note that this implies that $f(d) = 3$ and $f(e)=5$, which in turn implies that $f(a) = 6$ and $f(b) = 2$. Because $a,b\in N(h'_1)$, we get $f(h'_1)=1$ as we wanted to prove.
\begin{figure}[ht]
	\centering
	\scalebox{.9}{
  \begin{tikzpicture}[scale=.8]
  \pgfsetlinewidth{1pt}
  \tikzset{vertex/.style={circle, minimum size=0.2cm, fill=black, draw, inner sep=1pt}}

\node[vertex,fill=black!50,label=180:{$h'_1 (1)$}] (h1) at (150:3.5) {};
\foreach \i in {2,6}{
\node[vertex,label=210-60*\i:{$h'_\i$}] (h\i) at (-\i*60+210:3.5) {};
}
\foreach \i in {3,4,5}{
\node[vertex,label=210-60*\i:{$h'_\i$},fill=white] (h\i) at (-\i*60+210:3.5) {$\i$};
}

\node[vertex,label=below:{$a (6)$}] (a1) at (168:2.2) {};
\node[vertex,label=above:{$b (2)$}] (a2) at (135:2.2) {};
\node[vertex,label=above:{$c(2)$}] (c) at (30:2.2) {};
\node[vertex,label={[above,xshift=.4cm] $d (3)$}] (d1) at (-15:2.1) {};
\node[vertex,label=270:{$e (5)$}] (d2) at (-45:2.2) {};
\node[vertex,label=left:{$g (6)$}] (e) at (-90:2.2) {};

\node[vertex,label=right:{\large $u$},fill=white] (v) at (0,0) {$1$};

\draw (h1)--(h2)--(h3)--(h4)--(h5)--(h6)--(h1)
        (v)--(a1) (v)--(a2) (v)--(c) (v)--(e)
        (h1)--(a1) (h1)--(a2)--(d1)--(c) (a1)--(d2)--(e)
        (h3)--(c) (h4)--(d1) (h4)--(d2) (h5)--(e);

\end{tikzpicture}
}
	\caption{Mapping of a diagonal gadget related to $u\in X$ such that $f(u)=1$, when $f(h'_i) =i $ for each $i\in\{3,4,5\}$. We get that $f(h'_1)=1$.}
	\label{fig:lemma1}
\end{figure}
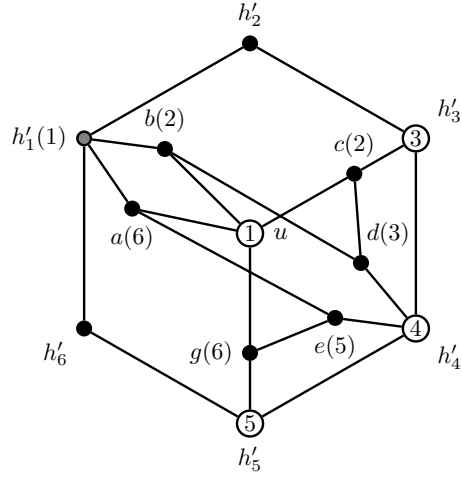
\end{proof}

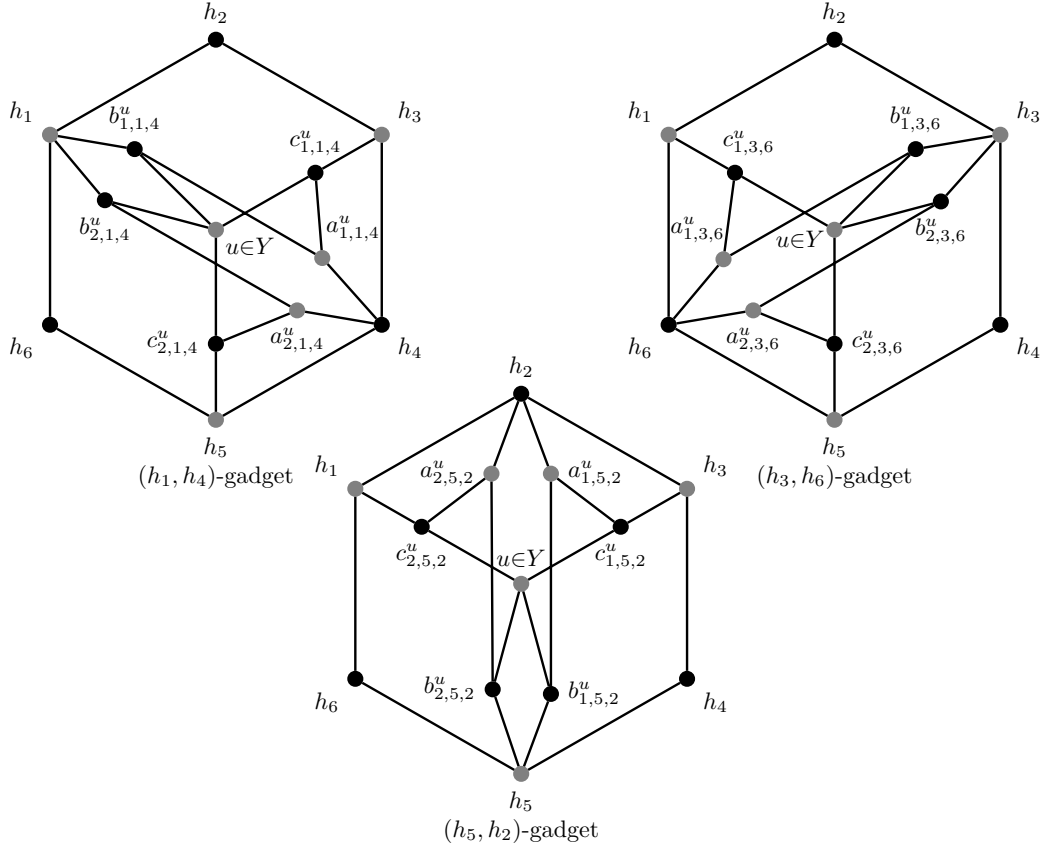
\begin{figure}[ht]
\noindent\begin{minipage}{\textwidth}
\begin{minipage}{0.49\textwidth}
\begin{flushleft}
\scalebox{.9}{
  \begin{tikzpicture}[scale=0.7]
  \pgfsetlinewidth{1pt}
  \tikzset{vertex/.style={circle, minimum size=0.2cm, fill=black, draw, inner sep=1pt}}

\foreach \i in {2,4,6}{
\node[vertex,label=210-60*\i:{$h_\i$}] (h\i) at (-\i*60+210:4) {};
}
\foreach \i in {1,3,5}{
\node[vertex,label=210-60*\i:{$h_\i$},black!50] (h\i) at (-\i*60+210:4) {};
}

\node[vertex,label=below:{$b^u_{2,1,4}$}] (a2) at (165:2.4) {};
\node[vertex,label=above:{$b^u_{1,1,4}$}] (a1) at (135:2.4) {};
\node[vertex,label=above:{$c^u_{1,1,4}$}] (c) at (30:2.4) {};
\node[vertex,label={[above,xshift=.45cm]$a^u_{1,1,4}$},black!50] (d1) at (-15:2.3) {};
\node[vertex,label=270:{$a^u_{2,1,4}$},black!50] (d2) at (-45:2.4) {};
\node[vertex,label=left:{$c^u_{2,1,4}$}] (e) at (-90:2.4) {};

\node[vertex,label={[below, xshift=.5cm, yshift=-.1cm]$u{\in}Y$},black!50] (v) at (0,0) {};

\draw (h1)--(h2)--(h3)--(h4)--(h5)--(h6)--(h1)
        (v)--(a1) (v)--(a2) (v)--(c) (v)--(e)
        (h1)--(a1) (h1)--(a2)--(d2)--(e) (a1)--(d1)--(c)
        (h3)--(c) (h4)--(d1) (h4)--(d2) (h5)--(e);

\node at (0,-5.2) {$(h_1,h_4)$-gadget};

\end{tikzpicture}
}
\end{flushleft}
\end{minipage}\hfill\begin{minipage}{0.49\textwidth}

\begin{flushright}
\scalebox{.9}{
  \begin{tikzpicture}[scale=0.7]
  \pgfsetlinewidth{1pt}
  \tikzset{vertex/.style={circle, minimum size=0.2cm, fill=black, draw, inner sep=1pt}}

\foreach \i in {2,4,6}{
\node[vertex,label=210-60*\i:{$h_\i$}] (h\i) at (-\i*60+210:4) {};
}
\foreach \i in {1,3,5}{
\node[vertex,label=210-60*\i:{$h_\i$},black!50] (h\i) at (-\i*60+210:4) {};
}

\node[vertex,label={[above,xshift=-.38cm]$a^u_{1,3,6}$},black!50] (a1) at (-165:2.4) {};
\node[vertex,label=below:{$a^u_{2,3,6}$},black!50] (a2) at (-135:2.4) {};
\node[vertex,label=above:{$b^u_{1,3,6}$}] (b1) at (45:2.4) {};
\node[vertex,label=below:{$b^u_{2,3,6}$}] (b2) at (15:2.3) {};
\node[vertex,label={[above,xshift=.2cm]$c^u_{1,3,6}$}] (c1) at (150:2.4) {};
\node[vertex,label=0:{$c^u_{2,3,6}$}] (c2) at (-90:2.4) {};

\node[vertex,label={[below, xshift=-.5cm]$u{\in}Y$},black!50] (v) at (0,0) {};

\draw (h1)--(h2)--(h3)--(h4)--(h5)--(h6)--(h1)
        (v)--(b1) (v)--(b2) (v)--(c1) (v)--(c2)
        (b2)--(h3)--(b1)--(a1)--(h6)--(a2)--(b2)
        (a1)--(c1)--(h1) (a2)--(c2)--(h5);

\node at (0,-5.2) {$(h_3,h_6)$-gadget};

\end{tikzpicture}
}
\end{flushright}
\end{minipage}
\end{minipage}


\noindent\begin{minipage}{\textwidth}
\begin{center}
\vspace{-2cm}
\scalebox{.9}{
  \begin{tikzpicture}[scale=0.7]
  \pgfsetlinewidth{1pt}
  \tikzset{vertex/.style={circle, minimum size=0.2cm, fill=black, draw, inner sep=1pt}}

\foreach \i in {2,4,6}{
\node[vertex,label=210-60*\i:{$h_\i$}] (h\i) at (-\i*60+210:4) {};
}
\foreach \i in {1,3,5}{
\node[vertex,label=210-60*\i:{$h_\i$},black!50] (h\i) at (-\i*60+210:4) {};
}

\node[vertex,label=right:{$a^u_{1,5,2}$},black!50] (a1) at (75:2.4) {};
\node[vertex,label=left:{$a^u_{2,5,2}$},black!50] (a2) at (105:2.4) {};
\node[vertex,label=right:{$b^u_{1,5,2}$}] (b1) at (-75:2.4) {};
\node[vertex,label=left:{$b^u_{2,5,2}$}] (b2) at (-105:2.3) {};
\node[vertex,label=below:{$c^u_{1,5,2}$}] (c1) at (30:2.4) {};
\node[vertex,label=below:{$c^u_{2,5,2}$}] (c2) at (150:2.4) {};

\node[vertex,label=above:{$u{\in}Y$},black!50] (v) at (0,0) {};

\draw (h1)--(h2)--(h3)--(h4)--(h5)--(h6)--(h1)
      (h2)--(a1)--(b1)--(h5)--(b2)--(a2)--(h2)
      (a1)--(c1)--(h3) (a2)--(c2)--(h1)
      (b2)--(v)--(b1) (c1)--(v)--(c2);

\node at (0,-5.2) {$(h_5,h_2)$-gadget};
\end{tikzpicture}
}
\end{center}
\end{minipage}
\caption{Different types of diagonal gadgets.}
\label{fig:Diagonals}
\end{figure}

Recall that Theorem~\ref{theo:C6retract} states that the restriction of \RetF{C_6} used in the following lemma is $\NP$-complete.

\begin{lemma}\label{lemma:C6contraction}
Let $G=(X\cup Y,E)$ be a bipartite graph and $H\subseteq G$ be a subgraph isomorphic to $C_6$. Let $Y_H$ be the set $V(H)\cap Y$, and suppose that $Y_H$ dominates $X$ and that ${\sf dist}(h,y)\le 2$ for every $h\in Y_H$ and $y\in Y$. If $G'$ is obtained as in Lemma~
\ref{lemma:diagonal}, then $G'$ is a bipartite graph with ${\sf diam}(G')\le 4$, and the answer to the \RetF{C_6} instance is \yes\ if and only if the answer to $G'$ admits an edge-surjective homomorphism to $C_6$.
\end{lemma}

\begin{proof}
Write $H$ as before and suppose $Y_H=\{h_1,h_3,h_5\}$. Denote by $A,B,C$ the sets of vertices containing the vertices labeled with $a,b,c$ in the gadgets depicted in Figure~\ref{fig:Diagonals}, respectively. To see that $G'$ is bipartite, observe the coloring in grey and black in Figure~\ref{fig:Diagonals}, with $Y$ being gray. Now, denote by $(X',Y')$ the bipartition of $G'$ such that $X\subseteq X'$, and observe that $X' = X\cup B\cup C$ and $Y' = Y\cup A$. We prove that the same property holding for $Y_H$ in $G$ also holds in $G'$. The fact that ${\sf diam}(G')\le 4$ then follows by the same argument given in Corollary~\ref{cor:3PreExt}. We discuss each property separately:
\begin{enumerate}
    \item\label{i} $Y_H$ dominates $X'$: we know that $Y_H$ dominates $X$ by assumption. Also, one can verify in Figure~\ref{fig:Diagonals} that every vertex in $B\cup C$ is adjacent to some vertex in $\{h_1,h_3,h_5\}$;
    \item\label{ii} ${\sf dist}(h,y)\le 2$ for every $y\in Y'$ and $h\in Y_H$: we know that this holds when $y\in Y$ by assumption. Consider a vertex $a\in A$. It suffices to show that this holds when $a$ is within an $(h_1,h_4)$-gadget, since the other cases are symmetric. So let $u$ be such that $a$ is within the $(h_1,h_4)$-gadget related to $u$. If $a = a_{1,1,4}^u$, then $(a,b_{1,1,4}^u,h_1)$, $(a,c_{1,1,4}^u,h_3)$, and $(a,h_4,h_5))$ are paths of length 2 between $a$ and each $h\in \{h_1,h_3,h_5\}$, as we wanted to show. An analogous argument holds if $a = a_{2,1,4}^u$.
\end{enumerate}

Now, we prove the second part of the theorem. Let $f$ be a retraction of $G$ to $H$. Since $H$ is isomorphic to $C_6$, it suffices to extend $f$ to $G'$. Also, because the gadgets are symmetric, we just need to show how to extend $f$ to an $(h_1,h_4)$-gadget. So consider any $u\in Y$. Figure~\ref{fig:lemma1} tells us how to extend $f$ to the $(h_1,h_4)$-gadget related to $u$ when $f(u)=h_1$ (considering $h'_1=h_1$), while Figure~\ref{fig:extendingDiagGadget} shows how to do it when $f(u)\in \{h_3,h_5\}$. Because $u\in Y$ and $\{h_2,h_4,h_6\}\subseteq X$, these are the only options.

\begin{figure}[t]
\noindent\begin{minipage}{\textwidth}
\begin{minipage}{0.49\textwidth}
\begin{center}
\scalebox{.8}{
  \begin{tikzpicture}[scale=.7]
  \pgfsetlinewidth{1pt}
  \tikzset{vertex/.style={circle, minimum size=0.2cm, fill=black, draw, inner sep=1pt}}

\node[vertex,label=135:{$h_1$}] (h1) at (150:4) {};
\foreach \i in {2,6}{
\node[vertex,label=210-60*\i:{$h_\i$}] (h\i) at (-\i*60+210:4) {};
}
\foreach \i in {3,4,5}{
\node[vertex,label=210-60*\i:{$h_\i$}] (h\i) at (-\i*60+210:4) {};
}

\node[vertex,label=below:{$(h_2)$}] (a1) at (165:2.4) {};
\node[vertex,label=above:{$(h_2)$}] (a2) at (135:2.4) {};
\node[vertex,label={[above,xshift=-.2cm]$(h_2/h_4)$}] (c) at (30:2.4) {};
\node[vertex,label={[above, xshift = .4cm]$(h_3)$}] (d1) at (-15:2.3) {};
\node[vertex,label=270:{$(h_3)$}] (d2) at (-45:2.4) {};
\node[vertex,label=left:{$(h_4)$}] (e) at (-90:2.4) {};

\node[vertex,label=right:{$u$},fill=white] (v) at (0,0) {\small $h_3$};

\draw (h1)--(h2)--(h3)--(h4)--(h5)--(h6)--(h1)
        (v)--(a1) (v)--(a2) (v)--(c) (v)--(e)
        (h1)--(a1) (h1)--(a2)--(d1)--(c) (a1)--(d2)--(e)
        (h3)--(c) (h4)--(d1) (h4)--(d2) (h5)--(e);
\end{tikzpicture}
}
\end{center}
\label{fig:variable}
\end{minipage}\hfill\begin{minipage}{0.49\textwidth}
\begin{center}
\scalebox{.8}{
  \begin{tikzpicture}[scale=.7]
  \pgfsetlinewidth{1pt}
  \tikzset{vertex/.style={circle, minimum size=0.2cm, fill=black, draw, inner sep=1pt}}

\node[vertex,label=135:{$h_1$}] (h1) at (150:4) {};
\foreach \i in {2,6}{
\node[vertex,label=210-60*\i:{$h_\i$}] (h\i) at (-\i*60+210:4) {};
}
\foreach \i in {3,4,5}{
\node[vertex,label=210-60*\i:{$h_\i$}] (h\i) at (-\i*60+210:4) {};
}

\node[vertex,label=below:{$(h_6)$}] (a1) at (165:2.4) {};
\node[vertex,label=above:{$(h_6)$}] (a2) at (135:2.4) {};
\node[vertex,label=above:{$(h_4)$}] (c) at (30:2.4) {};
\node[vertex,label={[above, xshift=.4cm]$(h_5)$}] (d1) at (-15:2.3) {};
\node[vertex,label=270:{$(h_5)$}] (d2) at (-45:2.4) {};
\node[vertex,label=left:{$(h_4/h_6)$}] (e) at (-90:2.4) {};

\node[vertex,label=right:{$u$},fill=white] (v) at (0,0) {\small $h_5$};

\draw (h1)--(h2)--(h3)--(h4)--(h5)--(h6)--(h1)
        (v)--(a1) (v)--(a2) (v)--(c) (v)--(e)
        (h1)--(a1) (h1)--(a2)--(d1)--(c) (a1)--(d2)--(e)
        (h3)--(c) (h4)--(d1) (h4)--(d2) (h5)--(e);

\end{tikzpicture}
}
\end{center}
\label{fig:clause2}
\end{minipage}
\end{minipage}
    \caption{Extension of a retraction of $G$ to $H$ to a retraction of $G'$ to $H$, which is of course also an edge-surjective $C_6$-homomorphism of $G'$. }
    \label{fig:extendingDiagGadget}
\end{figure}
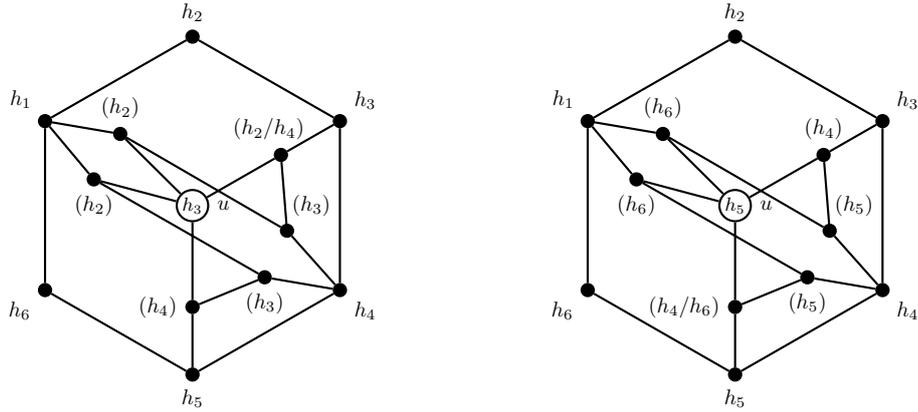

Conversely, suppose that $G'$ has an edge-surjective $C_6$-homomorphism $f$, and write the target $C_6$ as $H'=(1,\ldots,6)$. We want to prove that $f(h_i) = i$, for every $i\in [6]$, so that  $f$ also gives a retraction of $G$ to $H$. Note that if, at some point, we get that $f(h_i)=i$ either for every odd $i$, or for every even $i$, then we are done. 

First, we prove that $|f(Y_H)|>1$. So suppose without loss of generality that $f(Y_H) = 1$, and note that in this case $f(H)\subseteq \{1,2,6\}$. Because $G'$ is bipartite and $f$ is also vertex-surjective, we know that $f^{-1}(4)\neq\emptyset$, and since $f^{-1}(4)\subseteq X'$ we get a contradiction to Property (\ref{i}) of the set $Y_H$, since $X'$ would not be dominated by $Y_H$. We then may assume that $|f(Y_H)| = 2$, since the proof is finished when it is equal to~3. This means that 2 vertices among $\{h_1,h_3,h_5\}$ get distinct images. By relabeling $H'$ if necessary, we can assume that the possible cases are the following:

\begin{itemize}
    \item[$\bullet$] $f(h_3) = 3$ and $f(h_5) = 5$: this implies that $f(h_4)=4$. We know by Lemma~\ref{lemma:diagonal} that $f^{-1}(1)\cap Y=\emptyset$ as otherwise $f(h_1)=1$ and the lemma follows. Also, note that $\{h_2,h_4,h_6\}$ dominates $A$. Hence, since $\emptyset\neq f^{-1}(1)\subseteq A$, we must get that either $f(h_2)=2$ or $f(h_6)=6$. In fact, if we get that $f(h_2)=2$, then by relabeling $(1,\ldots,6)$ to $(3,4,5,6,1,2)$ and doing the same with $H$, we get the situation $f(h_i)=i$ for each $i\in \{3,4,5,6\}$. So suppose this is the case and note that we can assume that $f(h_1)=5$ and $f(h_2)=4$, since if $f(h_1)=1$ and $f(h_2)=2$ the proof is complete. We want to prove that $f^{-1}(1)=\emptyset$, thus getting a contradiction. Recall that $f^{-1}(1)\subseteq A$, and let $a\in A$ be such that $f(a)=1$. Also let $u\in Y$ be such that $a$ is within some diagonal gadget related to $u$. Observe that $a$ is not within an $(h_1,h_4)$-gadget since $\{a^u_{1,1,4},a^u_{2,1,4}\}\subseteq N(h_4)$, nor within an $(h_5,h_2)$-gadget since $\{a^u_{1,5,2},a^u_{2,5,2}\}\subseteq N(h_2)$, and $f(h_2)=f(h_4)=4$. Therefore, $a$ is within an $(h_3,h_6)$-gadget. The following argument is illustrated in Figure~\ref{fig:coloring3456}.

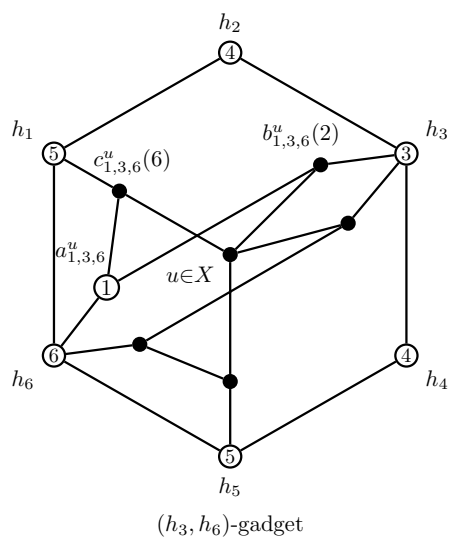
\begin{figure}[ht]
    \centering
    \scalebox{.85}{
    \begin{tikzpicture}[scale=.9]
\pgfsetlinewidth{1pt}
\tikzset{vertex/.style={circle, minimum size=0.2cm, fill=black, draw, inner sep=1pt}}

\foreach \i in {3,4,5,6}{
\node[vertex,scale=.9,label=210-60*\i:{$h_\i$},fill=white] (h\i) at (-\i*60+210:3.5) {$\i$};
}
\node[vertex,scale=.9,label=150:{$h_1$},fill=white] (h1) at (150:3.5) {$5$};
\node[vertex,scale=.9,label=90:{$h_2$},fill=white] (h2) at (90:3.5) {$4$};

\node[vertex,label={[above, xshift=-.4cm]$a^u_{1,3,6}$},fill=white] (a1) at (-165:2.2) {$1$};
\node[vertex,black] (a2) at (-135:2.2) {};
\node[vertex,label={[above,xshift=-.3cm]$b^u_{1,3,6}(2)$}] (b1) at (45:2.2) {};
\node[vertex] (b2) at (15:2.1) {};
\node[vertex,label={[above,xshift=.2cm]$c^u_{1,3,6} (6)$}] (c1) at (150:2.2) {};
\node[vertex] (c2) at (-90:2.2) {};

\node[vertex,label=210:{$u{\in}X$},black] (v) at (0,0) {};

\draw (h1)--(h2)--(h3)--(h4)--(h5)--(h6)--(h1)
        (v)--(b1) (v)--(b2) (v)--(c1) (v)--(c2)
        (b2)--(h3)--(b1)--(a1)--(h6)--(a2)--(b2)
        (a1)--(c1)--(h1) (a2)--(c2)--(h5);

\node at (0,-4.7) {$(h_3,h_6)$-gadget};

\end{tikzpicture}
    }%
    \caption{Situation where $f(h_i)=i$ for every $i\in \{3,4,5,6\}$.}
    \label{fig:coloring3456}
\end{figure}

        First suppose that $a = a^u_{1,3,6}$. Because $c^u_{1,3,6}\in N(h_1)$, $f(h_1)=5$, and $f(a)=1$, we get that $f(c^u_{1,3,6})=6$. And because $h_3\in N(b^u_{1,3,6})$, $f(h_3)=3$, and $f(a)=1$, we get that $f(b^u_{1,3,6})=2$.
    We get a contradiction since $\{2,6\}\subseteq f(N(u))$ and $f(u)\neq 1$. Clearly the same argument holds for $a=a^u_{2,3,6}$ since $f(h_5)=f(h_1)=5$.

    \item[$\bullet$] $f(h_1)=1$ and $f(h_3) = 3$: an argument analogous to the previous case yields a similar contradiction;
    \item[$\bullet$] $f(h_1) = 1$ and $f(h_5) = 5$: again, an analogous argument concludes the proof.
\end{itemize}
\vspace{-.45cm}
\end{proof}

As a direct consequence of Theorem~\ref{theo:C6retract} and Lemma~\ref{lemma:C6contraction}, we get that {\Comp} is $\NP$-complete even when $G$ has diameter four.  We next argue that this also implies the $\NP$-completeness of \textsc{Surjective $C_6$-Homomorphism}.

\begin{corollary}\label{cor:Comp_Hom}
{\Comp}$|_{{\cal B}_4}$ is $\NP$-complete. The same holds for \textsc{Surjective $C_6$-Homomorphism}$|_{{\cal B}_4}$.
\end{corollary}
\begin{proof}
Let $G$ be a bipartite graph with diameter at most four. The first statement follows by Theorem~\ref{theo:C6retract} and Lemma~\ref{lemma:C6contraction}. We argue that $G$ has an edge-surjective $C_6$-homomorphism if and only $G$ has a surjective $C_6$-homomorphism. Clearly an edge-surjective $C_6$-homomorphism is a surjective $C_6$-homomorphism, so it remains to show that given a surjective $C_6$-homomorphism $f$ of $G$, it follows that $f$ is an edge-surjective $C_6$-homomorphism as well. Suppose otherwise, and denote the target $C_6$ by $(1,\ldots,6)$. Without loss of generality, suppose that edge $12$ has no pre-image, i.e., there is no $uv\in E(G)$ such that $f(u)=1$ and $f(v)=2$. This means that every path between $u\in f^{-1}(1)$ and $v\in f^{-1}(2)$ is mapped to $(1,6,5,4,3,2)$, which implies that ${\sf dist}(u,v)\ge 5$, contradicting the hypothesis that ${\sf diam}(G) \leq 4$.
\end{proof}

Finally, by applying a trick similar to the one used in~Proposition~\ref{prop:complexityOrder}, we get that:

\begin{corollary}\label{cor:Comp_Hom2}
For an integer $k\ge 4$, let $M_k$ denote the complete bipartite graph $K_{k,k}$ minus a perfect matching. Also, let $G\in {\cal B}_3$ and $H\subseteq G$ such that $H\cong M_k$. Then  \textsc{Surjective $M_k$-Homomorphism},  \textsc{Edge-Surjective $M_k$-Homomorphism}, and \RetF{M_k} are all $\NP$-complete.
\end{corollary}

\section{\textsc{3-Biclique Partition} on bipartite graphs with diameter four}\label{sec:3biclique}

In this section we first show how Corollary~\ref{cor:Comp_Hom} implies the $\NP$-completeness of the \textsc{3-Biclique Partition} problem, and then we present the flaw in the proof in~\cite{FMPS.09}.
Recall that $\overline{G}_{\cal B}$ denotes the bipartite complement of $G$.

\begin{corollary}\label{cor:3biclique}
Let $G$ be a bipartite graph such that $\overline{G}_{\cal B}$ has diameter at most four. Then, deciding whether $G$ has a 3-biclique partition is $\NP$-complete.
\end{corollary}
\begin{proof}
Denote the parts of $G$ by $X$ and $Y$. It suffices to notice that $V_1,V_2,V_3$ is a 3-biclique partition of $G$ if and only if the function $f$ defined as follows is a surjective $C_6$-homomorphism of $\overline{G}_{\cal B}$: $f^{-1}(1) = X\cap V_1$, $f^{-1}(2) = Y\cap V_2$, $f^{-1}(3) = X\cap V_3$, $f^{-1}(4) = Y\cap V_1$, $f^{-1}(5) = X\cap V_2$, and $f^{-1}(6) = Y\cap V_3$. The theorem thus follows from Corollary~\ref{cor:Comp_Hom}.
\end{proof}

 Fleischner et al.~\cite{FMPS.09} presented a hardness proof for the  \textsc{$k$-Biclique Partition} problem for a general $k\ge 3$, but since their reduction for values of $k$ larger than~3 can be obtained using the trick presented in Proposition~\ref{prop:complexityOrder}, for the sake of simplicity we consider $k=3$. Still for the sake of simplicity, we work on the bipartite complement of their construction, which is from the {\LkCol{3}} problem. So, consider $G$ bipartite with parts $X,Y$, and a list assignment $L$ such that $L(u)\subseteq\{1,2,3\}$ for every $u\in V(G)$. Let $G'$ be obtained from $G$ by adding a cycle $C=(x_1,y_2,x_3,y_1,x_2,y_3,x_1)$ (which alternatively can be seen as the complete bipartite graph minus the perfect matching $\{x_iy_i\mid i\in [3]\}$). Then, for every $u\in X$, add an edge from $u$ to $y_i$ if and only if $i\notin L(u)$; do the same for each $u\in Y$ and $x_i$. They claim that $G$ has an $L$-coloring if and only if $G'$ has a retraction to $C$, if and only if $\overline{G'}_{\cal B}$ has a $3$-biclique partition. Indeed one can see that any retraction $f$ of $G'$ to ${C}$ gives a 3-biclique partition $V_1,V_2,V_3$ of  $\overline{G'}_{\cal B}$ with  each $V_i$ defined as $ \{u\in V(G)\mid f(u)\in \{x_i,y_i\}\}$. However, in the reverse implication, it is not necessarily true that a 3-biclique partition $V_1,V_2,V_3$ will map edge $x_iy_i$ inside of $V_i$ for every $i\in [3]$, as the authors claim. For example, when $G$ is simply an edge $uv$, and $L(u)=L(v)=\{1,2\}$, we get that $V_1=\{x_1,x_2,v\}$, $V_2 = \{y_1,y_2,u\}$, and $V_3 = \{x_3,y_3\}$ is a valid 3-biclique partition of $V(G')$.

Nevertheless, their proof does imply that \RetF{C_6} is $\NP$-complete. We provided another proof of this fact in Theorem~\ref{theo:C6retract} because we needed stronger constraints on $G$ and $H$ in order to prove Corollaries~\ref{cor:Comp_Hom} and~\ref{cor:3biclique}.


\section{\textsc{\texorpdfstring{$k$}{k}-Fall Coloring}}\label{sec:3fall}

In this section we investigate the complexity of \textsc{$k$-Fall Coloring} on bipartite graphs with diameter at most~$d$ for every pair of integers $k,d$. As in the case of list coloring problems, again the only case that we leave open is when $k=3$ and $d=3$. We conjecture that this case is also $\NP$-complete, and we prove in Proposition~\ref{prop:3falld3} that, if so, then the cases left open in Table~\ref{table:diam3} would also be $\NP$-complete.
We start with the following technical lemma.

\begin{lemma}\label{lem:tool}
Let $G$ be a bipartite graph with vertex bipartition $(X,Y)$ and $f$ be a $k$-fall-coloring of $G$. If $k\ge 3$, then $f(X)=f(Y)=[k]$.
\end{lemma}

\begin{proof}
    Towards a contradiction, suppose that $1 \notin f(X)$, which implies that every $v \in X$ has some neighbor in $f^{-1}(1) \subseteq Y$; however, there can be no $u \in Y$ with $f(u) \neq 1$, otherwise $1 \notin f(N[u])$, so we have $f^{-1}(1) = Y$.
    But in this case, since $k \geq 3$ and $G$ is bipartite, $|f(N[v])| \leq 2$ for every $v \in X$, contradicting the hypothesis that $f$ is a $k$-fall-coloring.
\end{proof}

The following is analogous to Proposition~\ref{prop:complexityOrder}.

\begin{proposition}\label{prop:fallk_k+1}
Let $k \geq 3$ be a fixed positive integer. Then,
\[\kFall{k}|_{{\cal B}}\preceq \kFall{(k+1)}|_{{\cal B}_3}.\]
\end{proposition}

\begin{proof}
Let $G\in {\cal B}$ with parts $X$ and $Y$, and $G'$ be obtained from $G$ by adding new vertices $x,y$ together with all edges from $x$ to vertices in $Y$ and all edges from $y$ to vertices in $X$.
This is the same graph as the one constructed in Proposition~\ref{prop:complexityOrder}, and thus we already know that $G'$ has diameter at most three.
Now, we prove that $G$ has a $k$-fall-coloring if and only if $G'$ has a $(k+1)$-fall-coloring. If $f$ is a $k$-fall-coloring of $G$, then let $f'$ be obtained from $f$ by coloring $x$ and $y$ with color $k+1$. Every vertex of $X$ and $Y$ is adjacent to the new color, so they continue to be b-vertices, and $x$ and $y$ are b-vertices by Lemma~\ref{lem:tool}.
Now, let $f'$  be a $(k+1)$-fall-coloring of $G'$, and suppose without loss of generality that $f(x) = k+1$. Again, by Lemma~\ref{lem:tool} we get that each part must contain every color. Therefore, because $x$ is complete to $Y$, we get that the only vertex on $Y\cup \{y\}$ that can be colored with $k+1$ is $y$, i.e., $f(y) = k+1$. In this case, one can see that $f'$ restricted to $G$ must define a $k$-fall-coloring of $G$.
\end{proof}

We get the following partial classification of the problem. As we already mentioned, the only open case is when $k=3$ and $d=3$.

\begin{corollary}\label{cor:4Call}
Let $k,d$ be positive integers. Then $\kFall{k}|_{{\cal B}_d}$ is polynomial when $k\le 2$ or $d=2$, and $\NP$-complete when either $k\ge 4$ and $d\ge 3$, or $k= 3$ and $d\ge 4$.
\end{corollary}
\begin{proof}
Observe that if $G$ is a complete bipartite graph, i.e., a bipartite graph with  diameter two, then every coloring that uses more than $2$  colors will have a non-b-vertex, hence the answer to \kFall{k} is trivially \no\ when $k\ge 3$ and $G$ is a complete bipartite graph (that is, $d=2$). When $k\le 2$, then either $G$ has an isolated vertex and the answer is \no, or it does not and the answer is \yes\ since any 2-coloring is also a $2$-fall-coloring. For $k\ge 4$, it is known that \kFall{3} is $\NP$-complete on bipartite graphs~\cite{LL.09}, which applying Proposition~\ref{prop:fallk_k+1} and induction on $k$ gives us that \kFall{k}$|_{{\cal B}_d}$ is also $\NP$-complete for every $d\ge 3$.
For the remaining case, we prove in Theorem~\ref{thm:4fallCol}  that $\kFall{3}|_{{\cal B}_4}$ is $\NP$-complete.
\end{proof}

Before we move on to the proof of the case $k=3$ and $d=4$, we prove the following result.

\begin{proposition}\label{prop:3falld3}
$\kFall{3}|_{{{\cal B}_3}}\preceq \kPre{3}|_{{{\cal B}_3}}.$
\end{proposition}
\begin{proof}
We present a Turing reduction from $\kFall{3}|_{{\cal B}_3}$ to $\kPre{3}|_{{\cal B}_3}$, that is, we show that if $\kPre{3}|_{{\cal B}_3}$ can be solved in polynomial time, then we can solve $\kFall{3}|_{{\cal B}_3}$ by solving a polynomial number of instances of $\kPre{3}|_{{\cal B}_3}$. Let $G$ be a bipartite graph with diameter at most three. Given a cycle $C$ of length~6 and a 3-coloring $f$ of $G$, we say that $C$ is \emph{fall-colored in $f$} if $f$ restricted to $C$ is a fall-coloring. We claim that $G$ has a 3-fall-coloring if and only if there exists a $C_6$ in $G$ whose 3-fall-coloring can be extended to a proper 3-coloring of $G$. Observe that, if true, we get the desired reduction since it would suffice to test, for every subset of vertices of size~6 that induce a cycle $C$, whether a 3-fall-coloring of $C$ can be extended to $G$ (observe that this 3-fall-coloring is unique up to relabeling).

Let $(X,Y)$ be the bipartition of $G$. First, suppose that $f$ is a 3-fall-coloring of $G$. By Lemma~\ref{lem:tool}, we know that $f(X) = f(Y) = \{1,2,3\}$. Let $v_1,v_2,v_3\in X$ be colored with 1, 2, and 3, respectively. Because $G$ has diameter at most three, we get that $N(v_i)\cap N(v_j)\neq \emptyset$, for every $\{i,j\}\subseteq \{1,2,3\}$, $i\neq j$. So, let $w_{i,j}\in N(v_i)\cap N(v_j)$, for each choice of $i,j$. Since $f(v_1,v_2,v_3) = \{1,2,3\}$ and $f$ is a proper $3$-coloring, we get that $\{w_{1,2},w_{2,3},w_{1,3}\}$ are all distinct, and that $f(w_{i,j}) = \ell$ where $\ell\in \{1,2,3\}\setminus\{i,j\}$. Then the cycle $(v_1,w_{1,2},v_2,w_{2,3},v_3,w_{1,3})$ is an induced $C_6$ because $G$ is bipartite and because $f$ is a proper coloring.
Conversely, suppose that a 3-fall-coloring $f$ of a cycle $C$ of length~6 can be extended to a proper 3-coloring $f'$ of $G$ (note that the fall-coloring of $C$ is unique up to relabeling). A \emph{3-b-coloring} is a proper 3-coloring such that each color class has \emph{at least one} b-vertex. Note that any extension of $f$ is a 3-b-coloring of $G$, since in $C$ there are already b-vertices of all the~3 colors. Faik~\cite{F.05}  proved that every 3-b-coloring of a bipartite graph with diameter three is also a 3-fall-coloring\footnote{Since reference~\cite{F.05} is in French, for completeness we present the proof in Appendix~\ref{ap:Faik}.}; hence, $f'$ is a 3-fall-coloring of $G$.
\end{proof}

Now, we prove that \textsc{$3$-Fall-Coloring} is $\NP$-complete even when restricted to bipartite graphs with diameter at most four. We mention that our proof is an improvement on the proof presented by Laskar and Lyle~\cite{LL.09}, where the constructed graphs have diameter six, although the authors do not mention that in their proof.

\begin{theorem}\label{thm:4fallCol}
$\kFall{3}|_{{\cal B}_4}$ is $\NP$-complete.
\end{theorem}
\begin{proof}
We present a reduction from \textsc{3-Uniform 2-Col}. Consider a 3-uniform hypergraph $G$ on vertices $V = \{v_1,\dots,v_n\}$ and hyperedges $E=\{e_1,\ldots,e_m\}$, and let $G'$ be constructed as follows (see Figure~\ref{fig:FallCol} for an illustration). Add $V$ and $E$ to the set of vertices of $G'$, together with a copy $v'_i$ of each vertex $v_i\in V$; denote by $V'$ the set $\{v'_i\mid v_i\in V\}$. Also, add two new vertices $v,v'$, and make $v$ complete to $V$ and $v'$ complete to $V'$. Finally, add an edge between $e_j$ and each $v_i\in e_j$ for every $e_j\in E$, and add the matching $\{v_iv'_i\mid i\in[n]\}$. We prove that $G'$ is a bipartite graph with diameter four, and that $G$ is a \yes-instance of \textsc{3-Uniform 2-Col} if and only if $G'$ is a \yes-instance of \kFall{3}.

First, note that $(V'\cup E\cup \{v\}, V\cup\{v'\})$ is a bipartition of $G'$. To see that $G'$ has diameter four, first note that $G'-E$ consists of a perfect matching between $V$ and $V'$, together with a vertex $v$ complete to $V$ and a vertex $v'$ complete to $V'$. Observe that this subgraph has diameter three, with the most distant pairs of vertices being $v$ and $v'$, and $v_i$ and $v'_j$ with $i\neq j$. Now, consider a hyperedge $e\in E$. Below, we show that the distance between $e$ and any other vertex of $G'$ is at most four.

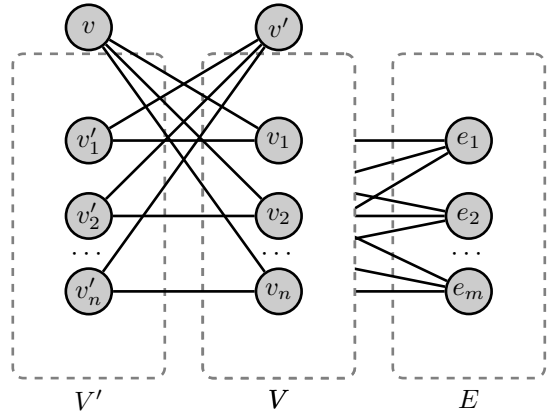
\begin{figure}[ht]
\begin{center}
  \begin{tikzpicture}[scale=1]
  \pgfsetlinewidth{1pt}

  \tikzset{vertex/.style={circle, minimum size=0.6cm, fill=black!20, draw, inner sep=1pt}}

    \node[vertex] (v1) at (0,4) {$v_1$};
    \node[vertex] (v2) at (0,3) {$v_2$};
    \node  at (0,2.5) {$\cdots$};
    \node[vertex] (vn) at (0,2) {$v_n$};

    \node[vertex] (vp1) at (-2.5,4) {$v'_1$};
    \node[vertex] (vp2) at (-2.5,3) {$v'_2$};
    \node  at (-2.5,2.5) {$\cdots$};
    \node[vertex] (vpn) at (-2.5,2) {$v'_n$};

    \node[vertex] (e1) at (2.5,4) {$e_1$};
    \node[vertex] (e2) at (2.5,3) {$e_2$};
    \node  at (2.5,2.5) {$\cdots$};
    \node[vertex] (em) at (2.5,2) {$e_m$};

    \node[vertex] (vp) at (0,5.5) {$v'$};
    \node[vertex] (v) at (-2.5,5.5) {$v$};

   \node [label=270:$V$, draw=black!50, rounded corners, dashed, minimum width=2cm, minimum height=4.3cm] at (0,3) {};
   \node [label=270:$V'$, draw=black!50, rounded corners, dashed, minimum width=2cm, minimum height=4.3cm] at (-2.5,3) {};
   \node [label=270:$E$, draw=black!50, rounded corners, dashed, minimum width=2cm, minimum height=4.3cm] at (2.5,3) {};

    \draw (e1)--(v1) (e1)--(0,3.3) (e1)--(0,2.5)
            (e2)--(0,3.5) (e2)--(0,3) (e2)--(0,2.5)
            (em)--(0,3.2) (em)--(0,2.5) (em)--(0,2);

\node [label=270:$V$, draw=black!50, rounded corners, dashed, minimum width=2cm, minimum height=4.3cm,fill=white] at (0,3) {};
\node[vertex] (v1) at (0,4) {$v_1$};
    \node[vertex] (v2) at (0,3) {$v_2$};
    \node  at (0,2.5) {$\cdots$};
    \node[vertex] (vn) at (0,2) {$v_n$};

    \draw (v)--(v1) (v)--(v2) (v)--(vn);
    \draw (vp)--(vp1) (vp)--(vp2) (vp)--(vpn);
    \draw (v1)--(vp1) (v2)--(vp2) (vn)--(vpn);
  \end{tikzpicture}
\end{center}
\caption{Graph $G'$ related to the hypergraph $G = (V,E)$.}
\label{fig:FallCol}
\end{figure}

\begin{itemize}
\item[$\bullet$] $d(e,v) = 2$: let $v_i\in e$; then $(e,v_i, v)$ is a path in $G'$;
\item[$\bullet$] $d(e,v') = 3$: let $v_i\in e$; then $(e,v_i,v'_i,v')$ is a path in $G'$;
\item[$\bullet$] $d(e,v_i)\le 3$ for every $v_i\in V$: if $v_i\in e$, then $(e,v_i)$ is a path in $G'$. Otherwise, let $v_j\in e$; then $(e,v_j,v,v_i)$ is a path in $G'$;
\item[$\bullet$] $d(e,v'_i)\le 4$ for every $v'_i\in V'$: if $v_i\in e$, then $(e,v_i,v'_i)$ is a path in $G'$. Otherwise, let $v_j\in e$; then $(e,v_j,v,v_i,v'_i)$ is a path in $G'$;
\item[$\bullet$] $d(e,e') \le 4$, for every $e'\in E(G)\setminus e$: if there exists $v_i\in e\cap e'$, then $(e,v_i,e')$ is a path in $G'$. Otherwise, let $v_i\in e$ and $v_j\in e'$, then $(e,v_i,v,v_j,e')$ is a path in $G'$.
\end{itemize}

Now, we prove that $G$ is a \yes-instance of \textsc{3-Uniform 2-Col} if and only if $G'$ is a \yes-instance of \kFall{3}. First, consider a 2-coloring $f$ of $G$ with no monochromatic hyperedge, and suppose that the used colors are $\{2,3\}$. We extend $f$ to a 3-fall-coloring $f'$ of $G'$. For this, color every $x\in E\cup\{v,v'\}$ with~1, and color $v'_i$ with $c\in \{2,3\}\setminus f(v_i)$. One can verify that, because no hyperedge of $G$ is monochromatic in $f$, the obtained coloring is a fall-coloring of $G'$.

Finally, consider a 3-fall-coloring $f'$ of $G'$, and suppose, without loss of generality, that $f'(v)=1$. This and the fact that $v$ is a b-vertex imply that $f(V) = \{2,3\}$. Hence, for every $e\in E(G)$, since $N_{G'}(e) \subseteq V$ and $1\notin f(V)$, in order for $e$ to be a b-vertex we must have that $f(e)=1$, and that $f(N_{G'}(e))=\{2,3\}$. Therefore, the coloring $f'$ restricted to $V$ is a 2-coloring of $G$ with no monochromatic hyperedge.
\end{proof}

\bibliographystyle{plain}

\begin{thebibliography}{}

\end{thebibliography}


\begin{thebibliography}{10}

\bibitem{AT.92}
Noga Alon and Michael Tarsi.
\newblock Colorings and orientations of graphs.
\newblock {\em Combinatorica}, 12(2):125--135, 1992.

\bibitem{Bacso.Tuza.90}
Gàbor Bacsó and Zsolt Tuza.
\newblock Dominating cliques in ${P}_5$-free graphs.
\newblock {\em Periodica Mathematica Hungarica}, 21:303--308, 1990.

\bibitem{BKM.12}
Manuel Bodirsky, Jan K{á}ra, and Barnaby Martin.
\newblock The complexity of surjective homomorphism problems - a survey.
\newblock {\em Discrete Applied Mathematics}, 160(12):1680--1690, 2012.

\bibitem{CHOSZ.20}
Maria Chudnovsky, Shenwei Huang, Paweł Rz{ą}żewski, Sophie Spirkl, and
  Mingxian Zhong.
\newblock Complexity of ${C}_k $-coloring in hereditary classes of graphs.
\newblock {\em arXiv preprint arXiv:2005.01824}, 2020.

\bibitem{CGKP.15}
Jean-Fran{ç}ois Couturier, Petr~A. Golovach, Dieter Kratsch, and Dani{ë}l
  Paulusma.
\newblock List coloring in the absence of a linear forest.
\newblock {\em Algorithmica}, 71(1):21--35, 2015.

\bibitem{CyganFKLMPPS15}
Marek Cygan, Fedor~V. Fomin, Lukasz Kowalik, Daniel Lokshtanov, D{á}niel Marx,
  Marcin Pilipczuk, Michal Pilipczuk, and Saket Saurabh.
\newblock {\em Parameterized Algorithms}.
\newblock Springer, 2015.

\bibitem{F.05}
Taoufik Faik.
\newblock {\em La b-continuité des b-colorations: complexité, propriétés
  structurelles et algorithmes}.
\newblock PhD thesis, Université de Paris Sud - U.F.R. Scientifique D'Orsay,
  2005.

\bibitem{FH.98}
Tomas Feder and Pavol Hell.
\newblock List homomorphisms to reflexive graphs.
\newblock {\em Journal of Combinatorial Theory, Series B}, 72:236--250, 1998.

\bibitem{FHH.99}
Tomas Feder, Pavol Hell, and Jing Huang.
\newblock List homomorphisms and circular arc graphs.
\newblock {\em Combinatorica}, 19(4):487--505, 1999.

\bibitem{FMPS.09}
Herbert Fleischner, Egbert Mujuni, Dani{ë}l Paulusma, and Stefan Szeider.
\newblock Covering graphs with few complete bipartite subgraphs.
\newblock {\em Theoretical Computer Science}, 410(21--23):2045--2053, 2009.

\bibitem{GJMPS.19}
Petr~A. Golovach, Matthew Johnson, Barnaby Martin, Dani{ë}l Paulusma, and
  Anthony Stewart.
\newblock {Surjective $H$-colouring: New hardness results}.
\newblock {\em Computability}, 8(1):27--42, 2019.

\bibitem{GMP.12}
Petr~A. Golovach, Bernard Lidick{ý}, Barnaby Martin, and Dani{ë}l Paulusma.
\newblock Finding vertex-surjective graph homomorphisms.
\newblock {\em Acta Informatica}, 49(6):381--394, 2012.

\bibitem{G.96}
Sylvain Gravier.
\newblock A {H}ajós-like theorem for list coloring.
\newblock {\em Discrete Mathematics}, 152:299--302, 1996.

\bibitem{H.61}
Gy{ö}rgy Hajós.
\newblock {Ü}ber eine konstruktion nicht n-f{ä}rbbarer graphen.
\newblock {\em Wiss. Z. Martin Luther Univ. Math-Natur. Reihe}, 10:116--117,
  1961.

\bibitem{HN.90}
Pavol Hell and Jaroslav Ne{š}et{ř}il.
\newblock On the complexity of \emph{H}-coloring.
\newblock {\em Journal of Combinatorial Theory, Series B}, 48(1):92--110, 1990.

\bibitem{HMSS.07}
Mohammad~H. Heydari, Linda Morales, Charles~O. Shields, and Ivan~H. Sudborough.
\newblock Computing cross associations for attack graphs and other
  applications.
\newblock In {\em Proc. of the 4th Hawaii International Conference on Systems
  Science (HICSS)}, page 270, 2007.

\bibitem{H.81}
Ian Holyer.
\newblock {The NP-completeness of edge-coloring}.
\newblock {\em SIAM Journal on Computing}, 10(4):718--720, 1981.

\bibitem{HKLSS.10}
Ch{í}nh~T. Ho{à}ng, Marcin Kami{ń}ski, Vadim Lozin, Joe Sawada, and Xiao
  Shu.
\newblock Deciding $k$-colorability of ${P}_5$-free graphs in polynomial time.
\newblock {\em Algorithmica}, 57(1):74--81, 2010.

\bibitem{HJP.15}
Shenwei Huang, Matthew Johnson, and Dani{ë}l Paulusma.
\newblock Narrowing the complexity gap for colouring $({C}_s, {P}_t)$-free
  graphs.
\newblock {\em The Computer Journal}, 58(11):3074--3088, 2015.

\bibitem{HT.96}
Mih{á}ly Hujter and Zsolt Tuza.
\newblock {Precoloring Extension 3: Classes of Perfect Graphs}.
\newblock {\em Combinatorics, Probability and Computing}, 5:35--56, 1996.

\bibitem{JS.97}
Klaus Jansen and Petra Scheffler.
\newblock Generalized coloring for tree-like graphs.
\newblock {\em Discrete Applied Mathematics}, 75(2):135--155, 1997.

\bibitem{K.11}
Marcin Kami{ń}ski.
\newblock Open problems from algorithmic graph theory, 2011.
\newblock Bled, Slovenia.

\bibitem{Karp72}
Richard~M. Karp.
\newblock {\em Reducibility among Combinatorial problems}, pages 85--103.
\newblock Springer US, 1972.

\bibitem{K.93}
Jan Kratochv{í}l.
\newblock Precoloring extension with fixed color bound.
\newblock {\em Acta Mathematica Universitatis Comenianae}, 62(2):139--153,
  1993.

\bibitem{KTV.02}
Jan Kratochv{í}l, Zsolt Tuza, and Margit Voigt.
\newblock On the $b$-chromatic number of graphs.
\newblock In {\em Proc. of the 28th International Workshop on Graph-Theoretic
  Concepts in Computer Science}, pages 310--320. Springer-Verlag, 2002.

\bibitem{LL.09}
Renu Laskar and Jeremy Lyle.
\newblock Fall colouring of bipartite graphs and cartesian products of graphs.
\newblock {\em Discrete Applied Mathematic}, 157:330--338, 2009.

\bibitem{L.73}
László Lovász.
\newblock Coverings and colorings of hypergraphs.
\newblock In {\em Proc. of the 4th Southwest Conference on Combinatorics},
  pages 3--12. Utilitas Math, Winnipeg, 1973.

\bibitem{MRS.91}
Nadimpalli~V.R. Mahadev, Fred~S. Roberts, and Prakash Santhanakrishnan.
\newblock 3-choosable complete bipartite graphs.
\newblock Technical report, Rutgers University, 1991.

\bibitem{MP.15}
Barnaby Martin and Dani{ë}l Paulusma.
\newblock The computational complexity of disconnected cut and
  $2{K}_2$-partition.
\newblock {\em Journal of Combinatorial Theory, Series B}, 111:17--37, 2015.

\bibitem{MS.16}
George~B. Mertzios and Paul~G. Spirakis.
\newblock {Algorithms and Almost Tight Results for 3-Colorability of Small
  Diameter Graphs}.
\newblock {\em Algorithmica}, 74(1):385--414, 2016.

\bibitem{Donnel}
Paul O'Donnel.
\newblock The choice number of ${K}_{6,q}$.
\newblock Preprint Rutgers Univ. Math. Dept.

\bibitem{P.16}
Dani{ë}l Paulusma.
\newblock Open problems on graph coloring for special graph classes.
\newblock In {\em Proc. of the 41st International Workshop on Graph-Theoretic
  Concepts in Computer Science (WG)}, volume 9224 of {\em LNCS}, pages 16--30,
  2016.

\bibitem{ST.95}
A.M. Shende and Barry Tesman.
\newblock 3-choosability of ${K}_{5,q}$.
\newblock {\em Congressus Numerantium}, 111:193--221, 1995.

\bibitem{V.99}
Narayan Vikas.
\newblock Computational complexity of compaction to cycles.
\newblock In {\em Proc. of the 10th annual ACM-SIAM Symposium on Discrete
  Algorithms (SODA)}, pages 977--978, 1999.

\bibitem{V.04}
Narayan Vikas.
\newblock Compaction, retraction, and constraint satisfaction.
\newblock {\em SIAM Journal on Computing}, 33(4):761--782, 2004.

\bibitem{V.17}
Narayan Vikas.
\newblock Computational complexity of graph partition under vertex-compaction
  to an irreflexive hexagon.
\newblock In {\em Proc. of the 42nd International Symposium on Mathematical
  Foundations of Computer Science (MFCS)}, volume~83 of {\em LIPIcs}, pages
  69:1--69:14, 2017.

\bibitem{V.76}
Vadim~G. Vizing.
\newblock Coloring the vertices of a graph in prescribed colors.
\newblock {\em Diskret. Analiz.}, 29(3):10, 1976.

\bibitem{W96}
Douglas~B. West.
\newblock {\em Introduction to graph theory}, volume~2.
\newblock Prentice Hall Upper Saddle River, NJ, 1996.

\end{thebibliography}

\begin{appendix}

\section{\texorpdfstring{$3$}{3}-b-colorings and \texorpdfstring{$3$}{3}-fall-colorings}
\label{ap:Faik}

In this section, for the sake of completeness, we present a proof of Faik~\cite{F.05}.

\begin{theorem}[Faik~\cite{F.05}]
Let $G$ be a bipartite graph with diameter at most~3. If $f$ is a $3$-b-coloring of $G$, then $f$ is a $3$-fall-coloring of $G$.
\end{theorem}
\begin{proof}
Let $(X,Y)$ be the bipartition of $V(G)$. By Lemma~\ref{lem:tool}, it holds that that $f(X) = f(Y) = [3]$. Note that if $u,v$ are within the same part, then $N(u)\cap N(v)\neq \emptyset$,
as otherwise their distance would be at least four. So, let $u\in X$ be of color~1. Because there exists $v\in X$ of color $2$ and since $N(u)\cap N(v)\neq \emptyset$, we get that $u$ must have a neighbor of color 3, namely the common neighbor with $v$. The analogous holds when picking any $v\in X$ of color~$3$; therefore $u$ is a b-vertex. Clearly this argument can be applied to every $u\in X\cup Y$ just by renaming the colors and the parts.
\end{proof}

\end{appendix}

\end{document}